\documentclass[11pt]{amsart}
\usepackage{format}
\title{On the affinization of a nilpotent orbit cover}

\begin{document}
	
	\begin{abstract}
		Let $\mathfrak{g}$ be a simple classical Lie algebra over $\mathbb{C}$ and $G$ be the adjoint group. Consider a nilpotent element $e\in \mathfrak{g}$, and the adjoint orbit $\mathbb{O}=Ge$. The formal slices to the codimension $2$ orbits in the closure  $\overline{\mathbb{O}}\subset \mathfrak{g}$ are well-known due to the work of Kraft and Procesi \cite{Kraft-Procesi}. In this paper we prove a similar result for the universal $G$-equivariant cover $\widetilde{\mathbb{O}}$ of $\mathbb{O}$. Namely, we describe the dimension $2$ singularities for its affinization $\Spec(\mathbb{C}[\widetilde{\mathbb{O}}])$. 
	\end{abstract}

	\maketitle
	
	\section{Introduction}
	Let $\fg$ be a classical simple Lie algebra, and $G$ be the corresponding adjoint group. We set $\OO\subset \fg$ to be a nilpotent orbit and $\widetilde{\OO}$ to be its universal $G$-equivariant cover. Moreover, let $\OO'\subset \overline{\OO}$ be a codimension $2$ orbit, and $\widetilde{\OO}'\subset \Spec(\CC[\widetilde{\OO}])$ be the preimage of $\OO'$ under the composition $\Spec(\CC[\widetilde{\OO}])\to \Spec(\CC[\OO])\to \overline{\OO}$. Here the first map is induced from the covering $\widetilde{\OO}\to \OO$, and the second is the normalization map. As we show, in many cases $\widetilde{\OO}'$ is connected. We note that if $\widetilde{\OO}'$ has several connected components, the formal slices to each of the connected components are isomorphic.
	
	Kraft and Procesi in \cite{Kraft-Procesi} described formal slices to $\OO'\subset \overline{\OO}$, we recall this result in \cref{KP}. The main goal of this paper is to provide two different modifications of this famous result. Inspired by results of \cite{fuetal2015}, in \cref{theorem_symmetry} we complete the description of the formal slice by adding a symmetry acting on the slice. The second and main goal of this paper is to describe the formal slice to (a connected component of) $\widetilde{\OO}'$ in $\Spec(\CC[\widetilde{\OO}])$. Note that if $\fg$ is of type A, then $\widetilde{\OO}=\OO$, and the description of formal slices is well-known. Therefore, we assume from now on that $\fg$ is classical of type B, C or D. We remark that our description includes the symmetry acting on the slice. As an important intermediate step we obtain explicit Lie-theoretic description of partial Namikawa spaces for nilpotent orbit covers. The key idea carried out in \cref{sect: partial} is to identify the partial Namikawa space corresponding to a codimension $2$ orbit $\OO'\subset \overline{\OO}$ with a center of a certain Levi subalgebra of $G$. Thus, we show that this Levi carries all essential information about the slice to $\widetilde{\OO}'$ in $\Spec(\CC[\widetilde{\OO}])$. In \cite[Section 7.5]{LMBM} this idea gave rise to the notion of an adapted Levi subalgebra. One can view \cref{sect: partial} as an explicit computation of adapted Levi subalgebras in classical types.
	
	The plan of the paper is as follows. 
	
	In \cref{nilpotent} we recall classical facts about nilpotent orbits and state the two main results: \cref{theorem_symmetry} and \cref{theorem}. After that we generalize the notion of Lusztig-Spatenstein induction to covers of the nilpotent orbits. In \cref{sp sing} we recall the description of filtered Poisson deformations of symplectic singularities from \cite{Namikawa}. Namely, filtered Poisson deformations of a symplectic singularity $X$ are classified by the points of $\fP/W$, where $\fP$ is an affine space, and $W$ is a finite group acting on $\fP$. The Namikawa space $\fP$ can be recovered from the smooth locus $X^{reg}$ and the types of singularities of codimension $2$ symplectic leaves in $X$. The main idea of our paper is to compute the space $\fP$ for $\Spec(\CC[\widetilde{\OO}])$, and deduce the codimension $2$ singularities from it. By results of \cite{Losev4}, the space $\fP$ for the conical symplectic singularity $\Spec(\CC[\OO])$ can also be described using the Lusztig-Spaltenstein induction. In \cref{quant} we generalize these results to the covers of nilpotent orbits. In \cref{sect: partial} we give an explicit Lie-theoretic description of all partial Namikawa spaces for an affinization of a nilpotent orbit cover and deduce \cref{theorem_symmetry}. 
	In \cref{affin} we finish the computation of partial Namikawa spaces for $\Spec(\CC[\widetilde{\OO}])$ and deduce the singularity of $\widetilde{\OO}'$ in $\Spec(\CC[\widetilde{\OO}])$, thus proving \cref{theorem}. Lastly, in \cref{sect: duality} we discuss how results of this paper fit into a conjectural symplectic duality between slices between nilpotent orbits.
	
	\subsection{Acknowledgements}
		I am very grateful to Ivan Losev for suggesting this problem, a lot of fruitful discussions and for numerous remarks that helped me to improve exposition. The first version of a paper contained a mistake in \cref{compute p} which impacted the main result. I want to thank Shilin Yu for pointing this mistake to me, for many fruitful discussions on the subject, and for motivating me to rewrite this paper. 

        Section \cref{sect: duality} and the description of symmetries acting on the slices are inspired by discussions with Amihay Hanany, Deshuo Liu and Jiakang Bao to whom I am wholeheartedly thankful.

	\section{Nilpotent orbits in classical simple Lie algebras}\label{nilpotent}
	\subsection{Partitions corresponding to nilpotent orbits}\label{part}
	
	Let $\fg$ be a classical simple Lie algebra of type B, C or D. We have the following combinatorial description of nilpotent orbits in $\fg$.
	
	\begin{prop}\label{orbits=partitions}\cite[Theorem 5.1.6]{CM}	
		(1) The nilpotent $O_n$-orbits in $\fs\fo_{n}$ are in one-to-one correspondence with partitions of $n$ in which every even part occurs with even multiplicity.
		
		(2) The nilpotent $Sp_{2n}$-orbits in $\fs\fp_{2n}$ are in one-to-one correspondence with partitions of $2n$ in which every odd part occurs with even multiplicity.
	\end{prop}
	
	If $n$ is odd, then $O_n=\ZZ/2\ZZ\times SO_n=\ZZ/2\ZZ\times \Ad(\fs\fo_n)$, so every $O_n$ orbit is a single $\Ad(\fs\fo_n)$ orbit. Suppose that $n$ is even, let $\alpha=(\alpha_1, \ldots, \alpha_n)$ be a partition of $n$ and let $\OO_{\alpha}$ be the corresponding $O_n$ orbit. The partition $\alpha$ is called {\itshape very even} if $\alpha$ satisfies the parity condition (1), and every element of $\alpha$ is even. For a very even $\alpha$ the orbit $\OO_{\alpha}$ is the union of two $\Ad(\fs\fo_n)$-orbits $\OO_{\alpha}^I$ and $\OO_{\alpha}^{II}$. If $\alpha$ is not very even, the corresponding orbit $\OO_{\alpha}$ is a single $\Ad(\fs\fo_n)$-orbit. 
	
	%
	%
	%
	%
	In the setting of this paper we work with $G$-equivariant covers of $\OO$, where $G$ is the corresponding adjoint group. For $e\in \OO$ let $(e,h,f)$ be an $\fs\fl_2$-triple in $\fg$. We set $Q=Z_G(e,h,f)$, and set $\pi_1^G(\OO)=Z_G(e,h,f)/Z_G(e,h,f)^\circ$ to be the $G$-equivariant fundamental group of $\OO$.  
	
	Let $\alpha=(\alpha_1, \alpha_2,\ldots,\alpha_n)$ be the partition corresponding to the orbit $\OO\subset \fg$. Put $a$ to be the number of distinct odd $\alpha_i$, and $b$ to be the number of distinct even non-zero $\alpha_i$. We have the following proposition.
	\begin{prop}\label{pi_1}\cite[Corollary 6.1.6]{CM}
				
		(1) If $\fg=\fs\fo_{2n+1}$, then $\pi_1^G(\OO)=(\ZZ/2\ZZ)^{a-1}$;
		
		%
		(2) If $\fg=\fs\fp_n$, then $\pi_1^G(\OO)=(\ZZ/2\ZZ)^{b}$ if all even parts have even multiplicity; otherwise $\pi_1^G(\OO)=(\ZZ/2\ZZ)^{b-1}$;
		
		(3) If $\fg=\fs\fo_{2n}$, then $\pi_1^G(\OO)=(\ZZ/2\ZZ)^{\max(0,a-1)}$ if all odd parts have even multiplicity; otherwise $\pi_1^G(\OO)=(\ZZ/2\ZZ)^{\max(0,a-2)}$.
	\end{prop}
	
%
%
	
	\subsection{Singularities in codimension $2$}
	
	In this section we recall the classification of codimension $2$ singularities in $\overline{\OO}$ from \cite{Kraft-Procesi}. 
	
	We are interested in nilpotent orbits $\OO'\subset \overline{\OO}$ of codimension $2$. Let $\alpha$, $\beta$ be the partitions corresponding to $\OO$ and $\OO'$ respectively. Suppose that the first $r$ rows and $s$ columns of $\alpha$ and $\beta$ coincide. Let ${\alpha'}$, ${\beta'}$ be the partitions obtained by erasing these rows and columns. We call the pair $(\alpha',\beta')$ the minimal degeneration of $(\alpha,\beta)$.
	
	\begin{example}
		Consider $\fg=\fs\fp_{30}$ and set 
		
		$\alpha=$\ytableaushort{}
		* {10, 8,4,3,3,1,1}
		* [*(magenta)]{10, 4,4,3,3,1,1}, $\beta=$\ytableaushort{}
		* {10, 6,6,3,3,1,1}
		* [*(magenta)]{10, 4,4,3,3,1,1}
		
		
		\vspace{5mm}
		
		Then the minimal degeneration of $(\alpha,\beta)$ is $\alpha'=$\ydiagram{4},         $\beta'$=\ydiagram{2,2}.
	\end{example}
	
	\begin{theorem}\label{KP}\cite{Kraft-Procesi}
		Let $(\alpha, \beta)$ be a pair of partitions corresponding to the pair of orbits $(\OO, \OO')$, where $\OO'\subset \overline{\OO}$ is of codimension $2$. Then $(\alpha, \beta)$ has one of the following minimal degenerations.
		
		(a) $\alpha'=(2)$, $\beta'=(1,1)$. The singularity of $\OO'$ in $\overline{\OO}$ is the Kleinian singularity of type $A_1$.
		
		(b) $\alpha'=(2k)$, $\beta'=(2k-2,2)$ for some $k>1$. The singularity of $\OO'$ in $\overline{\OO}$ is the Kleinian singularity of type $D_{k+1}$.
		
		(c) $\alpha'=(2k+1)$, $\beta'=(2k-1,1,1)$ for some $k>0$. The singularity of $\OO'$ in $\overline{\OO}$ is the Kleinian singularity of type $A_{2k-1}$.
		
		(d) $\alpha'=(2k+1,2k+1)$, $\beta'=(2k,2k,2)$ for some $k>0$. The singularity of $\OO'$ in $\overline{\OO}$ is the Kleinian singularity of type $A_{2k-1}$.
		
		(e) $\alpha'=(2k,2k)$, $\beta'=(2k-1,2k-1,1,1)$ for some $k>0$. The singularity of $\OO'$ in $\overline{\OO}$ is the union of two Kleinian singularities of type $A_{2k-1}$ transversally meeting at $0$.
	\end{theorem}
	\begin{example}
		In the example above we are in situation (b) with $k=2$. So the singularity of $\OO_{\beta}\subset \overline{\OO}_{\alpha}$ is the Kleinian singularity of type $A_{3}$.
	\end{example}
	\subsection{Kleinian singularities with symmetries}
        In this subsection we mostly follow the exposition of \cite[Section 1.4.2] {fuetal2015}. Let $\Gamma\subset SL_2 (\CC)$ be a finite subgroup, and $\Sigma=\CC^2/\Gamma$ be the corresponding Kleinian singularity. Let $\rho: \widetilde{\Sigma}\to \Sigma$ be the minimal resolution, and recall that the exceptional divisor $E$ of $\rho$ is a union of projective lines. The intersection graph of irreducible components of $E$ is the Dynkin diagram of a simply laced Dynkin diagram $\Delta$, which we call the Dynkin diagram of $\Gamma$. Any automorphism of $\Sigma$ fixes the point $0\in \Sigma$ and acts on the irreducible components of the exceptional divisor by permuting them, thus inducing a diagram automorphism of $\Delta$. Thus, we may consider the pairs $(\Sigma, K)$ of a Kleinian singularity $\Sigma$ and a subgroup $K$ of diagram automorphisms of $\Delta$. We say that $K$ is a symmetry acting on the singularity $\Sigma$.

        In \cite[III.6]{Slodowy} Slodowy introduced the following notations for the pairs $(\Sigma, K)$
        \begin{itemize}
            \item $B_k$ for $X$ of type $A_{2k-1}$ and $K=S_2$;
            \item $C_k$ for $X$ of type $D_{k+1}$ and $K=S_2$;
            \item $F_4$ for $X$ of type $E_6$ and $K=S_2$;
            \item $G_2$ for $X$ of type $D_4$ and $K=S_3$.
        \end{itemize}
        The notations are motivated by the following fact. If $\fg$ is a simple Lie algebra, then the slice to the subregular orbit in the nilpotent cone is a singularity of the same type as $\fg$. To simplify notations we formally set $B_0=A_1$.

        In \cite{fuetal2015} Fu, Juteau, Levy and Sommers described minimal singularities in nilpotent cones of exceptional Lie algebras and crucially described the symmetry acting on them. We can strengthen the results of \cref{KP} as follows.

        Let $(\alpha, \beta)$ be the partitions corresponding to $\OO$ and $\OO'$, and $(\alpha', \beta')$ be the minimal degeneration of $(\alpha, \beta)$. Let $q$ be the number of the row of $\alpha$ corresponding to the top-most row of $\alpha'$. We call such $q$ {\itshape the number of the singularity} of the pair $(\OO, \OO')$ and write it as $n(\OO, \OO')$. In the example above $q=2$.
        
        \begin{theorem}\label{theorem_symmetry}
           The singularity of $\OO'$ in $\overline{\OO}$ can be described as follows.
		
		(a) $\alpha'=(2)$, $\beta'=(1,1)$. The singularity of $\OO'$ in $\overline{\OO}$ is the Kleinian singularity of type $A_1$.
		
		(b) $\alpha'=(2k)$, $\beta'=(2k-2,2)$ for some $k>1$. If $\alpha_q$ and $\alpha_{q+1}$ are the only two odd members of $\alpha$ (automatically implying that $\fg=\fs\fo(2n)$), the singularity of $\OO'$ in $\overline{\OO}$ is the Kleinian singularity of type $D_{k+1}$. Otherwise, the singularity of $\OO'$ in $\overline{\OO}$ is of type $C_{k}$.
		
		(c) $\alpha'=(2k+1)$, $\beta'=(2k-1,1,1)$ for some $k>0$. The singularity of $\OO'$ in $\overline{\OO}$ is the singularity of type $B_{k}$.
		
		(d) $\alpha'=(2k+1,2k+1)$, $\beta'=(2k,2k,2)$ for some $k>0$. The singularity of $\OO'$ in $\overline{\OO}$ is the singularity of type $B_{k}$.
		
		(e) $\alpha'=(2k,2k)$, $\beta'=(2k-1,2k-1,1,1)$ for some $k>0$. The singularity of $\OO'$ in $\overline{\OO}$ is the union of two singularities of type $B_{k}$ transversally meeting at $0$.
        \end{theorem}
	\subsection{Main result}
	
	The main goal of this paper is to obtain an analogous result for codimension $2$ singularities of $\Spec(\CC[\widetilde{\OO}])$, where $\widetilde{\OO}$ is the universal $G$-equivariant cover of $\OO$ for $G$ being the adjoint group of $\fg$. The covering map $\widetilde{\OO}\to \OO$ extends to a $G$-equivariant map $\pi: \Spec(\CC[\widetilde{\OO}])\to \overline{\OO}$. We set $\widetilde{\OO}'=\pi^{-1}(\OO')$ and consider the singularity of $\widetilde{\OO}'$ in $\Spec(\CC[\widetilde{\OO}])$.
	
	    The following is the main result of this paper, describing the singularity of $\widetilde{\OO}'$ in $\Spec(\CC[\widetilde{\OO}])$ together with a symmetry acting on it. 
	
	\begin{theorem}\label{theorem}
		Let $\OO'\subset \overline{\OO}$ be a codimension $2$ orbit, and set $q=n(\OO,\OO')$. 
		
		(a) $\alpha'=(2)$, $\beta'=(1,1)$.  We have several cases. 
        \begin{itemize} 
            \item[(i)] If $\alpha_q\neq \alpha_{q-1}$ and $\alpha_{q+1}\neq \alpha_{q+2}$ are the only two even members of $\alpha$ (so that $\fg=\fs\fp(2n)$), $\widetilde{\OO}'$ is a disconnected $2$-fold cover of $\OO'$, and the singularity of each of the connected components of $\widetilde{\OO}'$ in $\Spec(\CC[\widetilde{\OO}])$ is the Kleinian singularity of type $A_{1}$.  
            \item[(ii)] If $\alpha_q$ and $\alpha_{q+1}$ are the only members of $\alpha$ with odd multiplicity, and not covered by (i), then $\widetilde{\OO}'$ is connected, and the singularity of $\widetilde{\OO}'$ in $\Spec(\CC[\widetilde{\OO}])$ is the Kleinian singularity of type $A_{1}$. 
            \item[(iii)] Otherwise, $\widetilde{\OO}'\subset \Spec(\CC[\widetilde{\OO}])^{reg}$.
        \end{itemize}
		(b) $\alpha'=(2k)$, $\beta'=(2k-2,2)$ for some $k>1$. We have several cases.
  \begin{itemize}
      \item[(i)] If $\alpha_q < \alpha_{q-1}$ and $\alpha_{q+1} > \alpha_{q+2}$ are the only two members of odd multiplicity in $\alpha$, then $\widetilde{\OO}'$ is connected, and the singularity of $\widetilde{\OO}'$ in $\Spec(\CC[\widetilde{\OO}])$ is the Kleinian singularity of type $D_{k+1}$.
      \item[(ii)] If $\alpha_q$ and $\alpha_{q+1}$ are the only members of $\alpha$ with odd multiplicity, but are not covered by (i), then $\widetilde{\OO}'$ is connected, and the singularity of $\widetilde{\OO}'$ in $\Spec(\CC[\widetilde{\OO}])$ is of type $C_{k}$.
      \item[(iii)] Otherwise, $\widetilde{\OO}'$ is connected, and the singularity of $\widetilde{\OO}'$ in $\Spec(\CC[\widetilde{\OO}])$ is of type $B_{k-1}$.
  \end{itemize}
		
		(c) $\alpha'=(2k+1)$, $\beta'=(2k-1,1,1)$ for some $k>0$. Then $\widetilde{\OO}'$ is connected, and the singularity of $\widetilde{\OO}'$ in $\Spec(\CC[\widetilde{\OO}])$ is of type $B_{k}$.
		
		(d) $\alpha'=(2k+1,2k+1)$, $\beta'=(2k,2k,2)$ for some $k>0$. Then $\widetilde{\OO}'$ is connected, and the singularity of $\widetilde{\OO}'$ in $\Spec(\CC[\widetilde{\OO}])$ is of type $B_{k}$.
		
		(e) $\alpha'=(2k,2k)$, $\beta'=(2k-1,2k-1,1,1)$ for some $k>0$. Then $\widetilde{\OO}'$ is connected, and the singularity of $\widetilde{\OO}'$ in $\Spec(\CC[\widetilde{\OO}])$ is of type $B_{k}$. 
	\end{theorem}	

	\begin{rmk}
		For the cases (a, ii) and (b, ii) we do not require the multiplicities of $\alpha_q$ and $\alpha_{q+1}$ to be equal $1$. For example, if $\OO$, $\OO'\subset \fs\fp_{22}$ are the orbits corresponding to the partitions $\alpha=(4,4,4,2,2,2,2,2)$ and $\beta=(4,4,3,3,2,2,2,2)$, then the singularity of $\widetilde{\OO}'$ in $\Spec(\CC[\widetilde{\OO}])$ is the Kleinian singularity of type $A_{1}$.
	\end{rmk}

    \begin{rmk}
		A similar phenomenon to (a, i) happens when $\alpha_q < \alpha_{q-1}$ and $\alpha_{q+1} > \alpha_{q+2}$ are the only two odd members of $\alpha$ (automatically  implying $\fg=\fs\fo(2n)$). In this case, there are two very even orbits corresponding to the partition $\beta$, that we denote by $\OO'^{I}$ and $\OO'^{II}$. Since $\pi_1(\OO)$ is trivial, $\widetilde{\OO}=\OO$, and the singularity to both $\OO'^{I}$ and $\OO'^{II}$ in $\Spec(\CC[\OO])$ is the Kleinian singularity of type $A_{1}$.
	\end{rmk}
	
	We have the following corollary of \cref{theorem}. It is, in particular, very useful in understanding almost \'etale maps and equivalences of covers introduced in \cite[Sections 5.4, 6.5]{LMBM}.
	\begin{cor}\label{etale}
		Let $\OO^{sreg}\subset \overline{\OO}$ be the union of $\OO$ and all codimension $2$ orbits $\OO'\subset \overline{\OO}$, such that one of the following two conditions holds.
		\begin{itemize}
			\item The pair of orbits $(\OO,\OO')$ is of types (c) or (d);
			\item $q=n(\OO,\OO')$ is such that $\alpha_q$ and $\alpha_{q+1}$ are the only members of $\alpha$ with odd multiplicity.
		\end{itemize}
		Then the map $\pi: \Spec(\CC[\widetilde{\OO}])\to \overline{\OO}$ is \'etale over $\OO^{sreg}$.  
	\end{cor}
	\subsection{Lusztig-Spaltenstein induction}\label{LS_sect}
	
	Let $\fl\subset \fg$ be a Levi subalgebra, and $\OO_L\subset \fl$ be a nilpotent orbit. We include $\fl$ into a parabolic subalgebra $\fp=\fl\ltimes \fn$, and set $P$ to be the corresponding parabolic subgroup of $G$. Following \cite{LS}, we consider the $G$-equivariant moment map $\mu: G\times^P(\overline{\OO}_L\times \fn)\to \fg$. The image contains a unique open $G$-orbit $\OO$. We say that $\OO$ is {\itshape induced} from $\OO_L$, and call $\mu$ the {\itshape generalized Springer morphism} for the pair ($\OO$, $\OO_L$). Moreover, if $\mu$ is birational then $\OO$ is called {\itshape birationally induced} from $\OO_L$ (see \cite{Losev4}). If the orbit $\OO$ cannot be birationally induced from a proper Levi subalgebra, we say that $\OO$ is {\itshape birationally rigid}.
	
	Let us explain the Lusztig-Spaltenstein induction on the level of partitions following \cite{CM}. Recall that $\fg$ is of type $B$, $C$ or $D$. We will use notation $\fg_n$ to denote the simple Lie algebra $\fs\fo_n$ or $\fs\fp_n$ respectively, and $G_n$ to denote the corresponding adjoint group. Denote by $\cP(n)$ the set of partitions appearing in \cref{orbits=partitions} for the Lie algebra $\fg_n$, and say that $l$ has the same parity as $\fg$, if $(-1)^l=\e_{\fg}$, where $\e_{\fg}=1$ for $\fg=\fs\fo_n$, and $\e_{\fg}=-1$ for $\fg=\fs\fp_n$. In other words, a partition $\alpha$ of $n$ corresponds to a nilpotent orbit (or two very even orbits in type $D$) in $\fg$ if and only if all parts of $\alpha$ of the same parity as $\fg$ occur with even multiplicity.
	
	Recall the following well-known fact. 
	\begin{lemma}\label{Levi}\cite[Lemma 3.8.1]{CM}
		Every Levi subalgebra $\fl \subset \fg_n$ is $G$-conjugate to a Levi subalgebra of the form $\fg\fl_{t_1}\times \fg\fl_{t_2}\times \ldots \times \fg\fl_{t_p}\times \fg_{n-2\sum_i t_i}$. Here we allow $n-2\sum_it_i$ is $0$ or $1$, which corresponds to the Levi $\fg\fl_{t_1}\times \fg\fl_{t_2}\times \ldots \times \fg\fl_{t_p}\subset \fg_n$.
	\end{lemma}
 
	By \cref{Levi}, any induction can be obtained as a sequence of inductions from the Levi subalgebras of form $\fl=\fg\fl_m\times \fg_{k-2m}\subset \fg_k$. Let us recall the combinatorial data associated with such induction. Let $\OO_L\subset \fg_{k-2m}$ be a nilpotent orbit. Formally $\OO_L$ is a nilpotent orbit in a simple Lie algebra $\fg_{k-2m}$, but to avoid clogging the formulas with the zero orbits we always identify it with the nilpotent orbit $\{0\}\times \OO_L\subset \fg\fl_m\times \fg_{n-2m}=\fl$. Let $\alpha\in \cP(k-2m)$ be the corresponding partition. Let $\OO\subset \fg$ be an orbit induced from the orbit $\OO_L\subset \fl$, and $\beta\in \cP(n)$ be the corresponding partition. Consider the partition $\alpha^m$ of $n$ defined by $\alpha^m_i=\alpha_i+2$ for $i\le m$ and $\alpha^m_i=\alpha_i$ for $i>m$. 
	\begin{prop}\label{LS}\cite[Theorem 7.3.3]{CM}
		There are two possible cases:
		
		1) $\alpha^m\in \cP(n)$. Then $\beta=\alpha^m$.
		
		2) $\alpha^m\notin \cP(n)$. Then $\beta=(\alpha_1+2, \ldots, \alpha_{m-1}+2, \alpha_{m}+1, \alpha_{m+1}+1, \alpha_{m+2}, \ldots, \alpha_n)$. We call $\beta$ the {\itshape collapse} of $\alpha^m$. Then $\alpha_m=\alpha_{m+1}$, and $\alpha_m$ has the same parity as $\fg$. Consequently, $\beta_m=\beta_{m+1}$, and $\beta_m$ has the opposite parity to $\fg$.
	\end{prop}
	
	\begin{example}
		Consider the orbit $\OO_L\subset \fs\fo_{11}$ corresponding to the partition $\alpha=(7,2,2)$. Let $\OO$ be the orbit in $\fs\fo_{15}$ induced from $\{0\}\times \OO_L\subset \fg\fl_2\times \fs\fo_{11}$. Then $\alpha^2=(9,4,2)\notin \cP(15)$, because even numbers $4$ and $2$ occur once. The orbit $\OO$ corresponds to the collapse $\beta=(9,3,3)$.
	\end{example}
	
	We can extend the Lusztig-Spaltenstein induction to the covers of orbits in the following way. Let $\OO_L\subset \fl$ be an adjoint orbit, $L\subset G$ be the Levi subgroup with Lie algebra $\fl$, and $\widehat{\OO}_L$ be an $L$-equivariant cover of $\OO_L$. Consider the $G$-equivariant moment map $\mu: G\times^P(\Spec(\CC[\widehat{\OO}_L])\times \fn)\to \fg$. The image contains a unique open $G$-orbit $\OO$ that is induced from $\OO_L$. Set $\widehat{\OO}=\mu^{-1}(\OO)$. Since $\mu$ is $G$-equivariant, $\widehat{\OO}$ is a $G$-equivariant cover of $\OO$. We say that $\widehat{\OO}$ is {\itshape birationally induced} from $\widehat{\OO}_L$. If the cover $\widehat{\OO}$ cannot be birationally induced from a proper Levi subalgebra, we say that $\widehat{\OO}$ is {\itshape birationally rigid}. 
	
%
	
	\section{Generalities on symplectic singularities and their deformations}\label{sp sing}
	\subsection{Filtered Poisson deformations}\label{quantizations}
	Let $A$ be a finitely generated Poisson algebra equipped with a grading $A=\bigoplus_{i=0}^\infty A_i$ such that $A_0=\CC$, and the Poisson bracket has degree $-d$, where $d\in \ZZ_{>0}$. 
	
	Let $\cA$ be a filtered commutative algebra equipped with a Poisson bracket decreasing the filtration degree by $d$. Moreover, suppose that we have an isomorphism $\theta: \gr \cA\simeq A$ of Poisson algebras. Such pair $(\cA, \theta)$ is called a {\itshape{filtered Poisson deformation}} of $A$.
	
	By a Poisson scheme we mean a scheme $X$ over $\Spec(\CC)$ whose structure sheaf $\cO_X$ is equipped with a Poisson bracket. For example, for any finitely generated Poisson algebra $A$ the affine scheme $X=\Spec(A)$ is a Poisson scheme. By a filtered Poisson deformation of $X$ we understand a sheaf $\calD$ of filtered Poisson algebras complete and separated with respect to the filtration together with an isomorphism of sheaves of Poisson algebras $\theta: \gr \calD\simeq \cO_X$. 
	
	
	
	\subsection{Symplectic singularities}
	
	Let $X$ be a normal Poisson algebraic variety such that the regular locus $X^{reg}$ is a symplectic variety. Let $\omega^{reg}$ be the symplectic form on $X^{reg}$. Suppose that there exists a projective resolution of singularities $\rho: \widehat{X}\to X$ such that $\rho^*(\omega^{reg})$ extends to a regular (not necessarily symplectic) form on $X$. Following Beauville \cite{Beauville2000}, we say that $X$ has symplectic singularities. Recall that we have the following example.
	\begin{example}\label{affin or orbit}\cite{Panyushev1991}
		Let $\fg$ be a semisimple Lie algebra and $\OO\subset \fg^*$ be a nilpotent orbit. Then $X=\Spec(\CC[\OO])$ has symplectic singularities.
	\end{example}
	
	We have the following generalization of \cref{affin or orbit}. The proof appeared earlier in \cite[Lemma 2.5]{LosevHC}).
	\begin{prop}\label{sympl_sing}	
		Let $\widehat{\OO}$ be a connected $G$-equivariant cover of $\OO$. The variety $\widehat{X}=\Spec(\CC[\widehat{\OO}])$ has symplectic singularities.
	\end{prop}
	
	We say that an affine Poisson variety $X$ is {\itshape conical} if $\CC[X]$ is endowed with a grading $\CC[X]=\bigoplus_{i=0}^\infty \CC[X]_i$ such that $\CC[X]_0=\CC$, and there exist a positive integer $d$ such that for any $i,j$ and $f\in \CC[X]_i$, $g\in \CC[X]_j$ we have $\{f,g\}\in \CC[X]_{i+j-d}$. Note that we have gradings on $\CC[\OO]$, $\CC[\widehat{\OO}]$, such that $\Spec(\CC[\OO])$ and $\Spec(\CC[\widehat{\OO}])$ are conical with $d=2$. For $\CC[\widehat{\OO}]$ see discussion before Theorem $1$ in \cite{Brylinski}.
	
	\subsection{The Namikawa space $\fP$}\label{sect: Namikawa}
	
	Recall that a normal algebraic variety $\widetilde{X}$ is called $\QQ$-factorial if for any Weil divisor it has a nonzero integral multiple that is Cartier. 
	\begin{prop}\cite[Proposition 2.1]{LosevSRA}
		Let $X$ have symplectic singularities. Then there is a birational projective morphism $\rho: \widetilde{X}\to X$, where $\widetilde{X}$ has the following properties:
		
		(a) $\widetilde{X}$ is an irreducible, normal, Poisson variety (and hence has symplectic singularities).
		
		(b) $\widetilde{X}$ is $\QQ$-factorial.
		
		(c) $\widetilde{X}$ has terminal singularities.
		
		If $X$ is, in addition, conical, then $\widetilde{X}$ admits a $\CC^\times$-action such that $\rho$ is $\CC^\times$-equivariant.
	\end{prop}
\begin{rmk}\label{Namikawa codim 4}
	In \cite{Namikawa2001} Namikawa has shown that modulo (a), condition (c) is equivalent to $\codim_{\widetilde{X}}\widetilde{X}/\widetilde{X}^{reg}\ge 4$. In what follows, we will check this condition instead of verifying (c).
\end{rmk}	
	 Such $\widetilde{X}$ is called a $\QQ$-factorial terminalization of $X$. We recall the {\itshape Namikawa space} for $X$ that is defined as $\fP(X)=H^2(\widetilde{X}^{reg}, \CC)$. When the conical symplectic singularity $X$ in question is clear, we write $\fP$ for $\fP(X)$. 
	
	We want to restate the definition of $\fP$ in terms of $X$ without using a $\QQ$-factorial terminalization. Let $\fL_1, \fL_2, \ldots, \fL_k$ be the codimension $2$ symplectic leaves of $X$. Let $\Sigma_i^\wedge=(\CC^2)^{\wedge 0}/\Gamma_i$ be the formal slice to $\fL_i$. Let $\Sigma_i$ be $\CC^2/\Gamma_i$, $\widetilde{\Sigma}_i$ be its minimal resolution, and $\widetilde{\fP}_i=H^2(\widetilde{\Sigma}_i, \CC)$. Recall that we have an ADE classification of finite subgroups of $Sp(2,\CC)$, and by McKay correspondence the affine space $\widetilde{\fP}_i$ is isomorphic to the Cartan space of the Lie algebra of the same type. Let $\widetilde{W}_i$ be the corresponding Weyl group acting on $\widetilde{\fP}_i$. We have the natural monodromy action of $\pi_1(\fL_i)$ on $\widetilde{\fP}_i$ and $\widetilde{W}_i$ by diagram automorphisms. Set $\fP_i=\widetilde{\fP}_i^{\pi_1(\fL_i)}$, $W_i=\widetilde{W}_i^{\pi_1(\fL_i)}$ to be the fixed points of this action. We have the following description of the space $\fP$ due to Namikawa \cite{Namikawa} and Losev \cite{Losev4}.
	\begin{prop}\label{P=P}
		We have $\fP=H^2(X^{reg}, \CC)\oplus \bigoplus_{i=1}^k \fP_i$.
	\end{prop}
	We call $\fP$ the {\itshape Namikawa space}, $\fP_0=H^2(X^{reg}, \CC)$ the {\itshape smooth part of the Namikawa space}, and $\fP_i$ the {\itshape partial Namikawa space} for the symplectic leaf $\fL_i$. 
	Recall the Namikawa-Weyl group, that is $W=\prod_{i=1}^k W_i$. The importance of $\fP$ and $W$ is explained by the following fact.
	\begin{prop}\label{general quant}\cite{Namikawa2}
		The isomorphism classes of filtered Poisson deformations of $\CC[X]$ are classified by the points of $\fP/W$.
	\end{prop}

	\section{Filtered Poisson deformations of $\Spec(\CC[\widehat{\OO}])$}\label{quant}
	\subsection{$\QQ$-terminalization of $\Spec(\CC[\widehat{\OO}])$}
	
	In \cite{Losev4} Losev computed the Namikawa space $\fP$ for $X=\Spec(\CC[\OO])$. In this section we extend his result to $\Spec(\CC[\widehat{\OO}])$ for any $G$-equivariant cover $\widehat{\OO}$ of the orbit $\OO$.
	
	\begin{prop}\label{Q-term}
		Suppose that $\widehat{\OO}$ is birationally rigid and let $X:=\Spec(\CC[\widehat{\OO}])$. Then $X$ is $\QQ$-factorial and terminal.
	\end{prop}
	\begin{rmk}	\label{Q-term to Namikawa}
		For the trivial cover $\widehat{\OO}\simeq \OO$ the proposition is known by the works of Namikawa \cite{Namikawa2009} and Losev \cite{Losev4}. 
		
		In \cite[Corollary 2.4, Proposition 3.4, Proposition 3.6]{NamikawaQ} Namikawa proved that $\Spec(\CC[\widetilde{\OO}])$ is $\QQ$-factorial terminal for the universal covering $\widetilde{\OO}$ of an orbit $\OO$ with some conditions on the corresponding partition $\alpha$. One can check that if $\widetilde{\OO}$ is birationally rigid then $\alpha$ satisfies the conditions from loc.cit.
	\end{rmk}
	\begin{proof}We follow the approach of \cite[Proposition 4.5]{Losev4}, this argument can be easily generalized to the case of $G$-equivariant covers. First, note that $X^{reg}\backslash\widehat{\OO}$ is of complex codimension at least $2$ in $X^{reg}$, so $H^2(X^{reg}, \CC)=H^2(\widehat{\OO}_L,\CC)$.
		%
		%
		Suppose that $\fP=0$. Let $x\in \widehat{\OO}$ be a point, and $Z_G(x)$ be the stabilizer of $x$. As in loc.cit. $H^2(\widehat{\OO}, \CC)=0$ implies $\Hom(Z_G(x), \CC^\times)$ is finite, so by \cite[Lemma 4.1]{Fu} $X$ is $\QQ$-factorial. If $\fP=0$, then $X$ has no symplectic leaves of codimension $2$, and by \cref{Namikawa codim 4} $X$ is terminal. It remains to show that $\fP=0$.
		
		In \cite{Namikawa} Namikawa constructed a universal conical Poisson deformation $X_\fP$ of $X$ over $\fP/W$. Pick Zariski generic $\lambda\in \fP/W$, and let $X_\lambda$ be the fiber over the point $\lambda$, and $\mu_\lambda: X_\lambda\to \fg$ be the moment map. The $G$-action on $X_\lambda$ has an open orbit that we denote by $\widehat{\OO}_\lambda$, and the restriction of $\mu_\lambda$ to the open orbit is a covering map $\widehat{\OO}_\lambda\to \OO_\lambda$ for some adjoint orbit $\OO_\lambda$. Let $\xi$ be the semisimple part of an element in $\OO_\lambda$ and set $\fl$ to be the centralizer of $\xi$ in $\fg$. We set $L$ and $P$ to be the corresponding Levi and parabolic subgroups. The argument of Step 3 in \cite[Proof of Proposition 4.5]{Losev4} implies that $\xi\neq 0$, and therefore $L\subset G$ is a proper Levi subgroup. Let $\OO_L\subset \fl$ be a nilpotent orbit, such that $\xi+\OO_L\subset \OO_\lambda$. Choose $\eta\in \OO_L$ and write $\OO_{L,\lambda}$ for the $L$-orbit of $\xi+\eta\in \fl$. The orbit $\OO_\lambda$ is induced from $\OO_{L,\lambda}$.
		
		Note that by the construction, $\xi$ is semisimple, $\eta$ is nilpotent and $[\xi,\eta]=0$.
		Since $Z_G(\eta+\xi)\subset Z_G(\xi)=L$, we have $Z_L(\eta+\xi)=Z_G(\eta+\xi)$. Since $\pi_1^G(\OO_\lambda)=Z_G(\eta+\xi)/Z_G(\eta+\xi)^\circ$, it follows that $\pi_1^G(\OO_\lambda)=\pi_1^L(\OO_{L,\lambda})$, and any $G$-equivariant cover of $\OO_\lambda$ can be birationally induced from the corresponding $L$-equivariant cover of $\OO_{L,\lambda}$. 
		
		Therefore $\widehat{\OO}_\lambda$ is birationally induced from some cover $\widehat{\OO}_{L,\lambda}$. Since $Z_L(\eta+\xi)=Z_L(\eta)$, we can identify the components group $Z_L(\eta+\xi)/Z_L(\eta+\xi)^\circ$ and $Z_L(\eta)/Z_L(\eta)^{\circ}$. Let $\widehat{\OO}_L$ be the cover of $\OO_L$ corresponding to the cover $\widehat{\OO}_{L,\lambda}$. 
		
		We set $X_L=\Spec(\CC[\widehat{\OO}_L])$, and let $\widetilde{X}_L$ be the $\QQ$-terminalization of $X_L$. Consider $\widetilde{X}_{\CC\xi}=G\times^P(\CC\xi\times \widetilde{X}_L\times \fn)$. That is a normal Poisson scheme over $\CC\xi$ with a Hamiltonian $G$-action, and the fiber over $0$ is $\widetilde{X}^1=G\times^P (\widetilde{X}_L\times \fn)$. Let $X_{\CC\xi}=\Spec(\CC[\widetilde{X}_{\CC\xi}])$. Note that the fiber of $X_{\CC\xi}$ over the point $\xi$ is $\Spec(\CC[\widehat{\OO}_{L,\lambda}])=X_{\lambda}$. The fiber over $0$ is  $X^1=\Spec(\CC[\widetilde{X}^1])$. The open $G$-orbit in $X^1$ is a finite $G$-equivariant cover $\check{\OO}$ of $\OO$, and $X^1=\Spec(\CC[\check{\OO}])$. Since $X_{\CC\xi}$ and $X_{\fP}$ are flat deformations, we have $\CC[X^1]\simeq \CC[X_{\lambda}]$ and $\CC[X_{\lambda}]\simeq \CC[X]$ as $G$-modules. Then $\check{\OO}=\widehat{\OO}$, and therefore $\widehat{\OO}$ is not birationally rigid, so we get a contradiction.
	\end{proof}
	\begin{cor}\label{P=z(l)}
		There is a unique pair of a Levi algebra $\fl$ and a birationally rigid cover $\widehat{\OO}_L$ of an orbit $\OO_L\subset \fl$, such that $\widehat{\OO}$ is birationally induced from $\widehat{\OO}_L$. The pair $(\fl, \widehat{\OO}_L)$ is called the \emph{birational induction datum} of $\widehat{\OO}$. The space $G\times^P(\Spec(\CC[\widehat{\OO}_L])\times \fn)$ is a $\QQ$-terminalization of $\Spec(\CC[\widehat{\OO}])$, and the Namikawa space $\fP$ equals to $\fz(\fl)$.
	\end{cor}
	
	\begin{proof}
		The proof is completely analogous to \cite[Corollary 4.6 and Proposition 4.7]{Losev4}. The only new thing we have to show is that there exists a birational map $\rho: G\times^P(\Spec(\CC[\widehat{\OO}_L])\times \fn)\to \Spec(\CC[\widehat{\OO}])$. Consider the Stein factorization of the generalized Springer map $G\times^P(\Spec(\CC[\widehat{\OO}_L])\times \fn)\to {Z}\to \overline{\OO}$. We note that since $\mu: G\times^P(\Spec(\CC[\widehat{\OO}_L])\times \fn)\to \overline{\OO}$ is $G$-equivariant, ${Z}$ admits an action of $G$. Then $\pi: {Z}\to \overline{\OO}$ is a finite $G$-equivariant map, and ${Z}$ is affine and normal. It follows that ${Z}=\Spec(\CC[\pi^{-1}(\OO)])$. The map $\rho: G\times^P(\Spec(\CC[\widehat{\OO}_L])\times \fn)\to {Z}$ is birational, thus implying that $Z=\Spec(\CC[\widehat{\OO}])$.
	\end{proof}
	\begin{rmk}\label{P=z(l) to Namikawa}
		For the universal cover $\widetilde{\OO}$ of $\OO$, a $\QQ$-terminalization of $\Spec(\CC[\widetilde{\OO}])$ is constructed in \cite{NamikawaQ}. This $\QQ$-terminalization is isomorphic to the one from \cref{P=z(l)}. 
	\end{rmk}
	\begin{cor}\label{restrict}
		Suppose that the cover $\widehat{\OO}$ is birationally induced from a birationally rigid cover $\widehat{\OO}_L$ of the orbit $\OO_L\subset \fl$ for some $\fl=\prod \fg\fl_{i\in S}^{k_i}\times \fg_{n'}$. Note that the only birationally rigid orbit in $\fg\fl_i$ is the $\{0\}$ orbit. Then $\widehat{\OO}$ cannot be birationally induced from $\{0\}\times \widehat{\OO}_{L'}\subset \fl'=\fg\fl_m\times \fg_{n-2m}$ for any $m\notin S$. 
	\end{cor}
	\begin{proof}
		Suppose that $\widehat{\OO}$ can be birationally induced from the cover $\widehat{\OO}_{L'}$ of an orbit $\{0\}\times \OO_{L'}\subset \fl'$. By \cref{P=z(l)} we have a Levi subalgebra $\fl_1\subset \fg_{n-2m}$ and a birationally rigid cover $\widehat{\OO}_{L_1}$ of an orbit $\OO_{L_1}\subset \fl_1$, such that $\widehat{\OO}_{L'}$ is birationally induced from $\widehat{\OO}_{L'}$. Then $\widehat{\OO}$ is birationally induced from $\widehat{\OO}_{L_1}$; that contradicts to the uniqueness of the pair $(\fl, \widehat{\OO}_L)$.
	\end{proof}

	\subsection{Birationality criteria for Lusztig-Spaltenstein induction}\label{birLS}
	In this section we study when the Lusztig-Spaltenstein induction on orbits is birational and use it to determine whether a given cover of an orbit can be birationally induced. Let us first introduce the notion of $\alpha$-singularity.
	
	Let $\OO\subset \fg$ be a nilpotent orbit, and $\alpha\in \cP(n)$ be the corresponding partition. For every $m\le n$ let $d_m=\left[\frac{\alpha_m-\alpha_{m+1}}{2}\right]$. We say that an integer $m$ is $\alpha$-{\itshape singular} if $\alpha_m-\alpha_{m+1}\ge 2$, or equivalently, $d_m\ge 1$. \cref{KP} implies that $\alpha$-singular numbers are in $1$-to-$1$ correspondence with partitions of codimension $2$ orbits in $\overline{\OO}$. We note that one $\alpha$-singular number may correspond to two very even orbits in type $D$, the partitions of these orbits are differed only by the indices $^{I}$, $^{II}$. The singularities to these two orbits are equivalent. We write $\OO_m$ for the union of $\OO_m^I\sqcup \OO_m^{II}$, and define the slice to $\OO_m\subset \overline{\OO}$ by considering the slice to any of the two connected components. In combinatorial terms, we ignore the indices of the partitions. We denote the set of $\alpha$-singular numbers by $\cS(\alpha)$. 
	\begin{lemma}\label{lll}
		Suppose that an orbit $\OO$ corresponding to the partition $\alpha\in \cP(n)$ is induced from a nilpotent orbit $\OO_L$ in the Levi subalgebra $\fl$, such that the restriction of $\OO_L$ to each $\fg\fl$-factor of $\fl$ is the $\{0\}$ orbit.
  
        Then $\fl$ is $G$-conjugate to $\prod_{i} \fg\fl_i^{k_i}\times \fg_{n-\sum_i 2ik_i}$, where $k_i\le d_i+1$. Moreover, $k_i=d_i+1$ only if $\alpha_{i-1}>\alpha_i=\alpha_{i+1}+2k_i>\alpha_{i+2}$, and $\alpha_i$ has the opposite parity to $\fg$.
	\end{lemma}
\begin{proof}
	By \cref{Levi}, $\fl$ is $G$-conjugate to $\prod_i \fg\fl_i^{k_i}\times \fg_{n-2\sum ik_i}$ for some $\{k_i\}$. Then $\OO$ is induced from an orbit $\OO_{L'}=\{0\}^{k_i}\times \OO_{L'}\subset \fg\fl_i^{k_i}\times \fg_{n-2ik_i}$, and let $\beta\in \cP(n-2ik_i)$ be the partition of $\OO_{L'}$. Then by \cref{LS} we have $\alpha_i-\alpha_{i+1}\ge (\beta_i+2k_i-1)-(\beta_{i}+1)=\beta_i-\beta_{i+1}+2k_i-2\ge 2k_i-2$.  Thus, $k_i\le d_i+1$. We note that $\alpha_i-\alpha_{i+1}< 2k_i$ only if inducing from $\OO_{L'}$ to $\fg_{n-2i(k_i-1)}$ involves collapse on the level of partitions. It implies that $k_i= d_i+1$ only if $\alpha_{i-1}>\alpha_i=\alpha_{i+1}+2k_i>\alpha_{i+2}$, and $\alpha_i$ has the opposite parity to $\fg$, q.e.d. 
\end{proof}
	
	We say that $\alpha$ is \textit{special at $k$}, if $\fg$ is of type $D$, and $\alpha_k=\alpha_{k+1}$ are the only odd parts of the partition $\alpha$. Recall the following well-known fact about the second cohomology of nilpotent orbits.
		
		\begin{prop} \label{coh}  
			Let $\OO\subset \fg$ be a nilpotent orbit in a classical simple Lie algebra. Let $\alpha$ be the corresponding partition.
			
			1) If $\alpha$ is special at $k$, then $H^2(\OO, \CC)=\CC$.
			
			2) If $\alpha$ is not special at any $k$, then $H^2(\OO, \CC)=0$.
		\end{prop}
		\begin{proof}
			 Let $\fz$ be the center of the Lie algebra $\fq$ corresponding to $\QQ$. We have $H^2(\OO, \CC)=(\fz^*)^Q$. A direct computation finishes the proof, see, for example, \cite[Theorems 5.4--5.6]{BISWAS}.
		\end{proof}
	
	\begin{prop}\label{induction}
		Let $\OO\subset \fg$ be a nilpotent orbit induced from $\{0\}\times \OO_L\subset \fg\fl_m\times \fg_{n-2m}$. We set $\alpha\in \cP(n-2m)$ to be the partition corresponding to $\OO_L$ and define $\beta$ as in \cref{LS}. Then $\OO$ is birationally induced from $\OO_L$ if and only if one of the following holds:
		
		1) $\beta=\alpha^m$;
		
		2) $\fg=\fs\fo_n$, $\beta$ is the collapse of $\alpha^m$, and all parts of $\alpha$ are even.
	\end{prop}
	\begin{proof}
		For $\fg=\fs\fp_{2n}$, this is proved in \cite[Claims 2.3.1,2.3.2]{NamikawaQ}. For $\fg=\fs\fo_n$, the reference is \cite[Claims 3.6.1,3.6.2]{NamikawaQ}.
	\end{proof}
	
	
	
	Let $\fl=\fg\fl_m\times \fg_{n-2m}$ be a Levi subalgebra, $\OO_L\subset \fg_{n-2m}$ be a nilpotent orbit, and $\rho: G\times^P (\overline{\OO}_L\times \fn)\to \overline{\OO}$ be the generalized Springer morphism. Set $\widehat{\OO}=\rho^{-1}(\OO)$. The embedding $i: \widehat{\OO}\to G\times^P ({\OO}_L\times \fn)$ induces a map $i_*: \pi_1(\widehat{\OO})\to \pi_1 (G\times^P ({\OO}_L\times \fn))$. Since both varieties are smooth, and the complement to $\widehat{\OO}$ in $G\times^P ({\OO}_L\times \fn)$ is of real codimension at least $2$, the map $i_{*}$ is surjective. We have a bundle $G\times^P({\OO}_L\times \fn)\to G/P$ with the fiber ${\OO}_L\times \fn$. The base $G/P$ is simply connected, so $\pi_1(G\times^P({\OO}_L\times \fn))=\pi_1({\OO}_L\times \fn)=\pi_1({\OO}_L)$. 
	\begin{prop}\label{pi1g}
		The map $i_*$ induces a well defined map $\phi: \pi_1^G(\widehat{\OO})\to \pi_1^{G_{n-2m}}(\OO_L)$. Moreover, $\phi$ is surjective.
	\end{prop}
	\begin{proof}
            This proposition proved in \cite[Section 2.5]{LMBM}. For the sake of completeness we provide the argument below.

            Pick a point $x\in \widehat{\OO}$, and let $G_x\subset G$ be the stabilizer of $x$.  Then the fibration $G\to G/G_x\simeq \widehat{\OO}$ induces an exact sequence
            $$\pi_1(G)\to \pi_1(\widehat{\OO})\to \pi_0(G_x)\to 0$$

            On the other side, consider $Z=G\times^P \OO_L$, where the action of the unipotent radical $U$ of $P$ on $\OO_L$ is trivial. We have a fibration $G\times^P ({\OO}_L\times \fn)\to Z$ with fiber $\fn$, which implies $\pi_1(\OO_L)\simeq \pi_1(G\times^P({\OO}_L\times \fn)) \simeq \pi_1(Z)$. By choosing a point $(1,y)\in Z$ the action of $G$ induces a fibration $G\to Z$ with the fiber $Q_{y}\simeq L_y\rtimes U$. Thewrefore, we have an exact sequence
            $$\pi_1(G)\to \pi_1(Z)\to \pi_0(Q_y)\simeq \pi_0(L_y)\to 0$$		

            Combining the two exact sequences with the identifications $\pi_1^G(\widehat{\OO})\simeq \pi_0(G_x)$, $\pi_1^L({\OO}_L)\simeq \pi_0(L_y)$, and $\pi_1(\OO_L)\simeq \pi_1(Z)$ we get the following commutative diagram.

		\begin{center}
    \begin{tikzcd}
      \pi_1(G) \ar[r] \ar[d,equals] & \pi_1(\widehat{\mathbb{O}}) \ar[r] \ar[d,twoheadrightarrow,"i_*"] & \pi_1^G(\widehat{\mathbb{O}}) \ar[r] & 1\\
      \pi_1(G) \ar[r] & \pi_1(\OO_L) \ar[r] & \pi_1^L({\mathbb{O}}_L) \ar[r] & 1
    \end{tikzcd}
\end{center}
    The claim follows immediately from the fact that $\pi_1^L(\OO_L)\simeq \pi_1^{G_{n-2m}}(\OO_L)$.
	\end{proof}
	
	\begin{cor}\label{o2}
		Suppose that the generalized Springer morphism $\rho: G\times^P (\overline{\OO}_L\times \fn)\to \overline{\OO}$ is not birational. Then $\widehat{\OO}$ is a $2$-fold cover of $\OO$, and $\dim H^2(\widehat{\OO},\CC)=\dim H^2(\OO,\CC)+1$.
	\end{cor} 
	\begin{proof}
		
		We have $\pi_1^G(\widehat{\OO})\subsetneq \pi_1^G(\OO)$, and therefore $|\pi_1^G(\OO)|\ge 2|\pi_1^G(\widehat{\OO})|$. By \cref{induction} and \cref{pi_1}, $|\pi_1^G(\OO)|=2|\pi_1^{G_{n-2m}}(\OO_L)|$. Since $\phi$ is an epimorphism, $|\pi_1^G(\widehat{\OO})|\ge |\pi_1^{G_{n-2m}}(\OO_L)|$. Therefore we have $|\pi_1^G(\OO)|=2|\pi_1^{G}(\widehat{\OO})|$, and $|\pi_1^G(\widehat{\OO})| = |\pi_1^{G_{n-2m}}(\OO_L)|$. Then $\pi_1(\widehat{\OO})\subset \pi_1(\OO)$ is a subgroup of index $2$, and $\widehat{\OO}\to \OO$ is $2$-fold. The map $\phi$ constructed above is an isomorphism.
		
		As in \cref{part}, let $Q$ be the reductive part of the stabilizer of $e\in \OO$, $\fq$ be the corresponding Lie algebra, and $\fz$ be the center of $\fq$, that is the direct sum of some copies of $\fs\fo_2$. It is well-known that $H^2(\OO, \CC)=(\fz^*)^Q=(\fz^*)^{\pi_1^G(\widehat{\OO})}$. Suppose that $\pi_1^{G_{n-2m}}({\OO}_L)\simeq (\ZZ/2\ZZ)^k$. The action of $\pi_1^G(\widehat{\OO})\simeq \pi_1^{G_{n-2m}}({\OO}_L)$ on $\fz^*$ is diagonalizable with $k$ eigenvalues $-1$, and all other $0$. Therefore, the space of invariants $(\fz^*)^{\pi_1^G(\widehat{\OO})}$ is the $0$ eigenspace for this action, so $\dim H^2(\widehat{\OO},\CC)=\dim (\fz^*)^{\pi_1(\widehat{\OO})}=\dim \fz-k$. The group $\pi_1^G(\OO)\simeq 2^{k+1}$ acts on $\fz^*$ with $k+1$ eigenvalues $-1$, so $\dim H^2(\OO,\CC)=\dim (\fz^*)^{\pi_1({\OO})}=\dim \fz - k-1$. Therefore, $\dim H^2(\widehat{\OO},\CC)=\dim H^2(\OO,\CC)+1$. 
	\end{proof}

    \begin{lemma}\label{preimage of phi}
        Let $\check{\OO}_L$ be an $G_{n-2m}$-equivariant cover of an orbit $\OO_L$, and $\check{\OO}$ be the $G$-equivariant cover of $\OO$ that is birationally induced from $(L, \check{\OO}_L)$. Then $\pi_1^G(\check{\OO})=\phi^{-1}(\pi_1^L(\check{\OO}_L)$.
    \end{lemma}
    \begin{proof}
        The following commutative diagram is Cartesian square.
        
    \begin{center}
    \begin{tikzcd}
      \check{\OO} \ar[r, hookrightarrow] \ar[d] & G\times^P (\check{\OO}_L\times \fn) \ar[d]\\
      \widehat{\OO} \ar[r, hookrightarrow]  & G\times^P ({\OO}_L\times \fn)
    \end{tikzcd}
    \end{center}
        This automatically implies that $\pi_1(\check{\OO})=i_*^{-1}(\pi_1(\check{\OO}_L))$. To obtain the claim in question we argue in the same way as in the proof above.
    \end{proof}
 
	Let us denote the subgroup $\pi_1^G(\widehat{\OO})\subset \pi_1^G(\OO)$ by $\Gamma_m$, and the kernel of the map $\phi$ by $H_m$. Note that $\Gamma_m=\pi_1^G(\OO)$ if $\rho$ is birational, and $H_m$ is trivial otherwise. 
	\begin{cor}\label{bircov}\label{birbircov}
		Let $\check{\OO}$ be a connected $G$-equivariant cover of $\OO$. $\check{\OO}$ is birationally induced from some $G_{n-2m}$-equivariant cover $\check{\OO}_L$ of $\OO_L$ if and only if $H_m\subset \pi_1^G(\check{\OO})\subset \Gamma_m$.
	\end{cor}
	\begin{proof}
		Suppose that $\check{\OO}$ is birationally induced from $\check{\OO}_L$. Then the covering map $\check{\OO}_L\to \OO_L$ induces a $G$-equivariant map $G\times^P(\Spec(\CC[\check{\OO}_L])\times \fn)\to G\times^P(\Spec(\CC[{\OO}_L])\times \fn)$, and therefore a covering map $\check{\OO}\to \widehat{\OO}$. Thus $\pi_1^G(\check{\OO})\subset \pi_1^G(\widehat{\OO})=\Gamma_m$. By \cref{preimage of phi} we have $\pi_1^G(\check{\OO})=\phi^{-1}(\pi_1^{G_{n-2m}}(\check{\OO}_L))$, so $H_m\subset \pi_1^G(\check{\OO})$.
		
		Now, suppose that $H_m\subset \pi_1^G(\check{\OO})\subset \Gamma_m$. Let $\check{\OO}_L$ be a cover of $\OO_L$ corresponding to the subgroup $\phi(\pi_1^G(\check{\OO}))\subset \pi_1^{G_{n-2m}}({\OO}_L)$. Let $\mathring{\OO}$ be the cover birationally induced from $\check{\OO}_L$. Then by \cref{preimage of phi} we have $\pi_1^G(\mathring{\OO})=\phi^{-1}(\pi_1^{G_{n-2m}}(\check{\OO}_L))=\pi_1^G(\check{\OO})$ as a subgroup of $\Gamma_m$, so $\mathring{\OO}\simeq \check{\OO}$.
	
	\end{proof}
%
    \section{Computing the partial Namikawa spaces}\label{sect: partial}
    Let us explain the main idea of this paper. Let $\widehat{\OO}$ be a connected $G$-equivariant cover of $\OO$, and let $\widehat{X}=\Spec(\CC[\widehat{\OO}])$. In most cases the singularity of $\widehat{\OO}_m\subset \Spec(\CC[\widehat{\OO}])$ can be determined by the dimension of the partial Namikawa space space $\fP_m=\fP_m(\widehat{X})$ corresponding to the (connected components of) $\widehat{\OO}_m$. Moreover, since $\fP_m=\widetilde{\fP}_m^{\pi_1(\widehat{\OO}_m)}$, this computation leads to understanding the symmetry of the singularity. In \cref{sect: Namikawa} the description of the spaces $\fP_m$ is given purely in terms of the slice to $\widehat{\OO}_m\subset \Spec(\CC[\widehat{\OO}])$. To achieve our results we need a Lie-theoretic description of $\fP_m$ that can be deduced from the partition $\lambda$ corresponding to $\OO$. In the language of \cite[Section 7.5]{LMBM} we explicitly compute the Levi subalgebra adapted to $\widehat{\OO}_m$. Let us state the main result, the rest of this technical section is devoted to the proof of it.
   
    Let $(\fl, \widehat{\OO}_L)$ be the birational induction datum corresponding to $\widehat{\OO}$. By \cref{Levi}, $\fl\simeq \prod_i \fg\fl_i^{t_i}\times \fg_{n-\sum_i 2it_i}$. \cref{lll} implies that we can rewrite it in the form $\fl\simeq \prod_{i\in S(\alpha)} \fg\fl_i^{k_i}\times \prod_{i\in S} \fg\fl_i\times \fg_{n-\sum_i 2it_i}$ for some finite set $S$, such that $S\cap S(\alpha)=\emptyset$. 
    \begin{prop}\label{combinatorics partial}
        Using the notations above, we have the following description of the partial Namikawa spaces of $\Spec(\CC[\widehat{\OO}])$.
        \begin{itemize}
            \item $\dim \fP_m=k_m$ for all $m\in S(\alpha)$;
            \item $\dim \fP_0=|S|$.
        \end{itemize}
    \end{prop}
    
    For the sake of clarity, we prefer to explain the result for $X=\Spec(\CC[\OO])$ first. It shows all the ideas of the general case in a setting, where we can avoid the overuse of diacritics. 
	
 \subsection{Computing the spaces $\fP_i$ from the partition $\alpha$}\label{Porbit}	
	Recall from \cref{birLS} the set of $\alpha$-singular elements $\cS(\alpha)$, and the numbers $d_i$ for all $i\le n$. For $m\in \cS(\alpha)$ let $\OO_m$ be the corresponding codimension $2$ orbit, and $\fP_m$ be the corresponding partial Namikawa space. When $\OO_m=\OO_m^I \sqcup \OO_m^{II}$ is the union of two very even orbits we define $\fP_m=\fP_m^I \sqcup \fP_m^{II}$, where $\fP_m^{*}$ is the partial Namikawa space corresponding to the symplectic leaf $\OO_m^{*}\subset \Spec(\CC[\OO])$.
 
    The first step to understand $\fP_m$ is to explicitly compute the Levi subalgebra $\fl$, such that $\OO$ is birationally induced from a birationally rigid nilpotent orbit in $\fl$. 
	
	\begin{prop}\label{prel}
		Let $t=n-2\sum_{i\in \cS(\alpha)} id_i$ and consider the Levi subalgebra $\fl_0=\prod_{i\in \cS(\alpha)} \fg\fl_i^{d_i}\times \fg_{t}\subset \fg$. There exists an orbit $\OO_{L_0}\subset \fl_0$, such that $\OO$ is birationally induced from $\OO_{L_0}$, and $\overline{\OO}_{L_0}$ has no codimension $2$ orbits.
	\end{prop}
	\begin{proof}
		Consider the partition $\beta=(\beta_1, \beta_2, \ldots, \beta_{t})$ obtained in the following way. For every number $k$ let $\cS(\alpha)_{\ge k}$ be the set of $m\in \cS(\alpha)$ such that $m\ge k$. Define $\beta_k=\alpha_k-2\sum_{i\in \cS(\alpha)_{\ge k}} d_i$. It is easy to see that $\beta\in \cP(t)$. We define $\OO_{L_0}$ to be the nilpotent orbit in $\fg_t$ and therefore in $\fl_0$ corresponding to the partition ${\beta}$. 
		
		The orbit $\OO$ is birationally induced from $\OO_{L_0}$ by \cref{induction}. Suppose that $\OO'_{L_0}\subset \overline{\OO}_{L_0}$ is an orbit of codimension $2$, corresponding to $m\in \cS(\beta)$. Then $\beta_m-\beta_{m+1}\ge 2$, and $\alpha_m-\alpha_{m+1}=(\beta_m+2\sum_{i\in \cS(\alpha)_{\ge m}} d_i) -(\beta_{m+1}+2\sum_{i\in \cS(\alpha)_{\ge m+1}} d_{i})=\beta_m-\beta_{m+1}+2d_m\ge 2d_m+2$. That contradicts to the definition of $d_m$. Therefore $\overline{\OO}_{L_0}$ has no codimension $2$ orbits. Note that this argument implies that $\OO_{L_0}$ cannot be very even orbit, and thus is uniquely defined by the partition $\beta$.
	\end{proof}
	\begin{cor}\label{l}
		Let $\fl_0$ and $\beta$ be as above. Define a Levi subalgebra $\fl\subset \fg$ as follows.
		If $\beta$ is special at $k$, then $\fl=\prod_{i\in \cS(\alpha)} \fg\fl_i^{d_i}\times \fg\fl_k\times \fg_{t-2k}\subset \fg$. Otherwise, $\fl=\fl_0$.
		
		There is a birationally rigid orbit $\OO_L\subset \fl$, such that $\OO$ is birationally induced from $\OO_L$.
	\end{cor}
	\begin{proof}
		If $\beta$ is not special at any $k$, then we set $\OO_L=\OO_{L_0}$. By \cref{prel}, $X_L=\Spec(\CC[\OO_L])$ has no symplectic leaves of codimension $2$, so the corresponding Namikawa space $\fP (X_L)$ equals $H^2(\OO_L, \CC)$. \cref{coh} implies that $\fP (X_L)=0$, and therefore $\OO_L$ is birationally rigid.
		
		If $\beta$ is special at $k$ then $\OO_{L_0}$ is birationally induced from some orbit $\OO_L\subset \fl$ by \cref{induction}. We have $\fP_0(X_{L_0})=H^2(\OO_{L_0}, \CC)=\CC$, and therefore $\dim \fP (X_L)=\dim \fP (X_{L_0}) -1=0$, so $\OO_L$ is birationally rigid.
	\end{proof}
	
	
	\begin{prop}\label{pifroml}
		Let $\OO$, $\lambda$, $\fl$ be as above, and consider $m\in \cS(\alpha)$. Let $\OO_m\subset \Spec(\CC[\OO])$ be the corresponding codimension $2$ orbit, and $\fP_m$ be the corresponding partial Namikawa space. Then:
		\begin{itemize}
                \item If $\alpha_m$ and $\alpha_{m+1}$ are the only two odd members of $\alpha$ (so $\fg$ is of type $D$), and both appear with multiplicity $1$, then $\fP_m\subset \fP$ can be identified with $\CC^{d_m+1}=\fz(\fg\fl_m^{d_m+1})\subset \fz(\fl)$.
			\item Otherwise, $\fP_m\subset \fP$ can be identified with $\CC^{d_m}=\fz(\fg\fl_m^{d_m})\subset \fz(\fl)$.
			\item If $\alpha$ is special at $k$ for some $k$, then $H^2(\OO, \CC)=\CC$ can be identified with $\fz(\fg\fl_k)\subset \fz(\fl)$.
			\item The action of $\pi_1(\OO_m)$ on $\widetilde{\fP}_m$ is trivial if and only if $d_m=1$ or $\alpha_m$ and $\alpha_{m+1}$ are the only two odd members of $\alpha$, both appearing with multiplicity $1$.
		\end{itemize}  
	\end{prop}
        We remark that if $\alpha_m$ and $\alpha_{m+1}$ are the only two odd members of $\alpha$, both appearing with multiplicity $1$, and $d_m=1$, then $\OO_m$ is the union of two very even orbits $\OO_m^I$, $\OO_m^{II}$. Then, $\fP_m=\fP_m^I\oplus \fP_m^{II}$ is identified with $\CC^2=\fz(\fg\fl_m^2)$, and singularity of each of the orbits is Kleinian singularity of type $A_1$, so the monodromy action is trivial. Thus, with our notations the proposition holds in this case.
        
	To prove the proposition we will use induction on $|\cS(\alpha)|$. We will need the following two lemmas.
	\begin{lemma}\label{singind}
		Consider a nilpotent orbit $\OO'\subset \fg_{n-2k}$, the corresponding partition $\alpha'$ and $m\in \cS(\alpha')$ such that $m\neq k$. Let $\OO'_m\subset \Spec(\CC[{\OO'}])$ be a codimension $2$ orbit corresponding to $m$. Suppose that $\OO$, $\OO_m\subset \fg_{n}$ are birationally induced from $\{0\times \OO'\}$, $\{0\times \OO'_m\}\subset \fg\fl_k\times \fg_{n-2k}$ respectively. Then the singularities of $\OO_m\subset\Spec(\CC[\OO])$ and $\OO'_m\subset \Spec(\CC[{\OO}'])$ are equivalent. 
	\end{lemma}
	\begin{proof}
		
	Let $\beta\in \cP(n)$ and $\beta'\in \cP(n-2k)$ be the partitions  corresponding to the orbits $\OO_m$ and $\OO'_m$ respectively. By \cref{KP} it is enough to show that the minimal degenerations of $(\alpha, \beta)$ and $(\alpha', \beta')$ give the same singularities. We say that $(\alpha, \beta)$ is of type (a)-(e) if the corresponding minimal degeneration is as in the respective case of \cref{KP}. We note that since we work with the normalization $\Spec(\CC[\OO])$ of $\overline{\OO}$, the minimal degeneration of $(\alpha, \beta)$ of type (e) leads to a singularity of type $A$.

        By \cref{induction}, either $\alpha=\alpha'^k$ or $\alpha$ is the collapse of $\alpha'^k$, and $\alpha'$ is very even. It implies that $\alpha_i=\alpha'_i+2$ if $i<k$, and $\alpha_i=\alpha'_i$ for $i>k+1$. Analogously, $\beta_i=\beta'_i+2$ if $i<k$, and $\beta_i=\beta'_i$ for $i>k+1$. Thus, the minimal degenerations coincide unless one of the following conditions holds.
        \begin{itemize}
            \item[(i)] $k=m-1$, and $(\alpha', \beta')$ is of type (d);
            \item[(ii)] $k=m-1$, and $(\alpha', \beta')$ is of type (e);
            \item[(iii)] $k=m+1$, and $(\alpha', \beta')$ is of type (c);
            \item[(iv)] $k=m+1$, and $(\alpha', \beta')$ is of type (e).
        \end{itemize}

        Let us fix notation $d'_m=\left[\frac{\alpha'_m-\alpha'_{m-1}}{2}\right]$. We analyze each case separately. 
        \begin{itemize}
            \item[(i)] Note that $\alpha'_{k}=\alpha'_{k+1}$ is of the same parity as $\fg$. Then \cref{induction} implies that $\fg=\fs\fo_n$, and $\alpha'$ is a very even orbit. However, $\alpha'_{k+2}$ is of different parity to $\fg$. We get a contradiction.
            \item[(ii)] Arguing as above, $\fg=\fs\fo_n$, and $\alpha'$ is a very even orbit. Direct computation shows that the minimal degeneration of $(\alpha, \beta)$ is $(2d'_m+1), (2d'_m-1, 1,1)$. Both minimal degenerations of $(\alpha, \beta)$ and $(\alpha', \beta')$ lead to the singularity of type $A_{2d'_m-1}$.
            \item[(iii)] $\alpha'_k$ is of the same parity as $\fg$. Then $\fg=\fs\fo_n$, and $\alpha'$ is a very even orbit. But $\alpha'_{k-1}$ is of the opposite parity to $\fg$. We get a contradiction.
            \item[(iv)] As above, $\fg=\fs\fo_n$, and $\alpha'$ is a very even orbit. Direct computation shows that the minimal degeneration of $(\alpha, \beta)$ is $(2d'_m+1, 2d'_m+1), (2d'_m, 2d'_m, 2)$. Both minimal degenerations of $(\alpha, \beta)$ and $(\alpha', \beta')$ lead to the singularity of type $A_{2d'_m-1}$.
        \end{itemize}

        To finish the proof we need to show that $\OO_m$ is a union of two very even orbits if and only if $\OO'_m$ is a union of two very even orbits.

        First, assume that $\beta$ is a very even partition. Then by \cref{induction} $\beta=\beta'^k$, and therefore all members of $\beta'$ must be even, thus $\beta'$ is a very even partition.

        Now assume that $\beta'$ is very even. If $\beta=\beta'^k$, it is very even. It remains to prove that $\beta=\beta'^k$. Assume the opposite, then $\beta'_{k-1}>\beta'_{k}=\beta'_{k+1}>\beta'_{k+2}$. We note that if $(\alpha', \beta')$ is of types (b)-(e), then $\beta'$ has members of the opposite parity to $\fg$, and thus cannot be very even. Therefore, $(\alpha', \beta')$ is of type (a), and $\alpha'$ is not very even. Since $\beta'_m=\beta'_{m+1}$, we have $m\neq k-1$, $k+1$. Thus,  $\alpha'_{k-1}=\beta_{k-1}>\alpha'_{k}=\beta'_k=\alpha'_{k+1}$, and $\alpha'^k\notin \cP(n)$. \cref{induction} implies that $\OO$ is not birationally induced from $\OO'$. We get a contradiction. 

	\end{proof} 

	The second lemma we need helps us to control the monodromy action of $\pi_1(\OO_m)$ on $\widetilde{\fP}_m$.
	\begin{lemma}\label{monodromy}
		Let $\OO'\subset \fg_{n-2k}$ be a nilpotent orbit, $\OO'_m\subset \overline{\OO}'$ be a codimension $2$ nilpotent orbit, $\widehat{\OO}'$ be a $G$-equivariant cover of $\OO'$, and $\mu': \Spec(\CC[\widehat{\OO}'])\to \fg$ be the moment map. Let $\widehat{\OO}'_m$ be a connected component of $\mu'^{-1}(\OO'_m)\subset \Spec(\CC[\widehat{\OO}'])$. Let $\widehat{\OO}$ be a cover of $\OO$ birationally induced from $\{0\}\times \widehat{\OO}'$ and the Levi subalgebra $\fl=\fg\fl_k\times \fg_{n-2k}\subset \fg$, and $\OO_m\subset \overline{\OO}$ be the orbit induced from $\{0\}\times \OO'_m\subset \fl$. Let $\widehat{\OO}_m\subset \Spec(\CC[\widehat{\OO}])$ be a connected component of the preimage $\mu^{-1}({\OO}_m)$ under the moment map $\mu: \Spec(\CC[\widehat{\OO}])\to \fg$. Suppose that $\widehat{\OO}_m$ is birationally induced from $\{0\}\times \widehat{\OO}'_m$ with the Levi subalgebra $\fl$. Moreover, assume that the singularity of $\widehat{\OO}_m\subset \Spec(\CC[\widehat{\OO}])$ is equivalent to the one of $\widehat{\OO}'_m\subset \Spec(\CC[\widehat{\OO}'])$.
		If the monodromy action of $\pi_1(\widehat{\OO}'_m)$ on $\widetilde{\fP'}_m$ is non-trivial, then the action of $\pi_1(\widehat{\OO}_m)$ on $\widetilde{\fP}_m$ is non-trivial. 
		
	\end{lemma}
	\begin{proof}
		Set $X=\Spec(\CC[\widehat{\OO}])$, $X'=\Spec(\CC[\widehat{\OO}'])$. Consider the map $\rho: Z=G\times^P(X'\times \fn)\to X$. Since $\widehat{\OO}_m$ is birationally induced from $\{0\times \widehat{\OO}'_m\}$, we have $\widehat{\OO}_m \subset G\times^P(\widehat{\OO}'_m\times \fn)\subset Z$. It follows that the singularity to $\widehat{\OO}_m$ in $Z$ is isomorphic to the singularity of $\widehat{\OO}'_m$ in $X'$. By the assumption of the lemma, the map $\rho$ is birational over $\widehat{\OO}\cup \widehat{\OO}_m\subset X$. 
  
  We have the following commutative diagram, where all maps are $G$-equivariant embeddings, and the middle horizontal map is an isomorphism.
		
		\begin{equation*}\label{emb}
			\xymatrix{
				& \widehat{\OO} \ar[d]\ar[rr]&& G\times^P(\widehat{\OO}'\times \fn) \ar[d]& \\
				&{Z}\ar[rr]&& G\times^P({X}'\times \fn)&\\
				& \widehat{\OO}_m\ar[rr]\ar[u]&& G\times^P(\widehat{\OO}'_m\times \fn)\ar[u]&}.
		\end{equation*}
		Analogously to \cref{o2}, the bottom map in the diagram induces a surjective map $\phi_m: \pi_1(\widehat{\OO}_m)\to \pi_1(\widehat{\OO}'_m)$, corresponding to the induction. 
    By the discussion above, there is a point $x=(x',n)\in \widehat{\OO}'_m \cap \widehat{\OO}'_m\times \fn$, such that the formal slice to $\widehat{\OO}'_m\subset X'$ at the point $x'$ is isomorphic to the formal slice to $\widehat{\OO}_m\subset X$ at the point $x$.
  
  Let $\widetilde{X}$ and $\widetilde{X}'$ be $\QQ$-factorial terminalizations of $X$ and $X'$ respectively, such that $\widetilde{X}\simeq G\times^P (\widetilde{X}'\times \fn)$. We note that results of \cite{Namikawa} imply that $\widetilde{X}$ and $\widetilde{X}'$ are uniquely defined over $\widehat{\OO}\cup \widehat{\OO}_m$ and $\widehat{\OO}'\cup \widehat{\OO}'_m$ respectively. Let us denote the partial resolution maps $\widetilde{X}\to X$, $\widetilde{X}'\to X'$ by $\rho_X$ and $\rho_{X'}$ respectively, and let $E_{\widetilde{X}}$ and $E_{\widetilde{X}'}$ be the corresponding exceptional divisors. We note that $E_{\widetilde{X}}$ and $E_{\widetilde{X}'}$ are closed under the action of $G$ and $L$ respectively. We have a commutative diagram, analogous to the one above.
		 	\begin{equation*}\label{emb2}
		 	\xymatrix{
		 		& \widehat{\OO} \ar[d]\ar[rr]&& G\times^P(\widehat{\OO}'\times \fn) \ar[d]& \\
		 		&\widetilde{X}\ar[rr]&& G\times^P(\widetilde{X}'\times \fn)&\\
		 		&\rho_X^{-1} (\widehat{\OO}_m) \ar[rr] \ar[u]&& G\times^P(\rho_{X'}^{-1} (\widehat{\OO}'_m) \times \fn)\ar[u]&}.
			\end{equation*}
	 
	 	
   Consider $x\in \widehat{\OO}_m$ and $x'\in \widehat{\OO}'_m$ as above. By the assumption of the lemma, we have $\widetilde{\fP}_m\simeq \widetilde{\fP}'_m$. Therefore the fibers $\rho_X^{-1}(x)$ and $\rho_{X'}^{-1}(x')$ are identified with the exceptional locus of the minimal resolution of the corresponding Kleinian singularities, that is a set of projective lines. 
	 	 
	 	Let $\gamma'\in \pi_1(\widehat{\OO}'_m, x')$ be an element such that $\gamma'$ acts non-trivially on the set of connected components of $\rho^{-1}(x')$. Pick $\gamma\in \phi_m^{-1}(\gamma)'\subset \pi_1(\widehat{\OO}_m, x)$. Then the action of $\widetilde{\gamma}$ intertwines the connected components of the fiber of the resolution map $f: \widetilde{X}=G\times^P(\widetilde{X}'\times \fn)\to G\times^P({X}'\times \fn)$ over the point $(1, x+n)$. We note that $\rho: \widetilde{X}\to X$ factors through $G\times^P({X}'\times \fn)$, and $(1, x+n)$ maps to $x$. Therefore ${\gamma}$ acts non-trivially on the connected components of $\rho_{X}^{-1}(x)$, so the action of $\pi_1(\widehat{\OO})$ on $\widetilde{\fP}_m$ is non-trivial.
	\end{proof}
	\begin{proofl}
            Let $\fl_0$, $\beta$, and $\fl$ be as in \cref{prel} and \cref{l} respectively. First, assume that $\cS(\alpha)=\{m\}$. 
		
		1) If $\alpha$ is not special at any $k$, then $\fP_0=H^2(\OO,\CC)=0$ by \cref{coh}, and therefore $\fP=\fP_m=\fz(\fg\fl_m^{k_m}\times \fg_{n-2mk_m})=\fz(\fg\fl_m^{k_m})$. By \cref{l}, $k_m=d_m+1$ if $\beta$ is special at $m$, and $k_m=d_m$ otherwise. The former is possible if and only if $\alpha_m$ and $\alpha_{m+1}$ are the only two odd members of $\alpha$, both appearing with multiplicity $1$. It implies the identifications of the proposition.
        
        Suppose that $d_m\ge 2$ and denote the corresponding Kleinian singularity by $S$. By \cref{KP}, either $S=A_{2d_m-1}$ and $\widetilde{\fP}_m=\CC^{2d_m-1}$ or $S=D_{4(d_m-1)}$ and $\widetilde{\fP}_m=\CC^{d_m+1}$. If $k_m=d_m$, the monodromy action must be non-trivial since $2d_m-1\ge d_m+1>d_m$. If $k_m=d_m+1$, then the singularity of $\OO_m\subset \Spec(\CC[\OO])$ is of type $D_{m+1}$, and the monodromy action is trivial. 
		
		
		2) If $\alpha$ is special at some $k$ then by \cref{induction} the orbits $\OO$ and $\OO_m$ can both be birationally induced from orbits $\OO'$ and $\OO'_m$ in $\fl=\fg\fl_k\times \fg_{n-2k}$ respectively. 
		Since $\alpha_k=\alpha_{k+1}$, we have $m\neq k$. We note that $\alpha$ being special at $k$ implies that $\beta$ is special at $k$, and therefore $\beta$ cannot be special at $m$. By \cref{singind}, the corresponding spaces $\widetilde{\fP}_m$ and $\widetilde{\fP}'_m$ are identified. If $d_m\ge 2$, then the action of $\pi_1(\OO'_m)$ is nontrivial by 1). By \cref{monodromy}, the action of $\pi_1(\OO_m)$ on $\widetilde{\fP}_m$ is nontrivial. If $d_m=1$, then $\widetilde{\fP}_m=\widetilde{\fP}'_m=\CC$, and the actions of $\pi_1(\OO_m)$ and $\pi_1(\OO'_m)$ are both trivial. Therefore, $\fP_m$ can be identified with $\fP'_m$. The latter one is identified with $\fz(\fg\fl_m^{d_m})$. We have $\fP=\fz(\fg\fl_k\times \fg\fl_m^{d_m}\times \fg_{n-2md_m-2k})=\fP_m\oplus \fz(\fg\fl_k)$. On the other hand, $\fP=H^2(\OO,\CC)\oplus \fP_m$, so $H^2(\OO,\CC)$ is identified with $\fz(\fg\fl_k)$.
		
		Suppose that we have proved Proposition for all $\OO=\OO_{\gamma}$ such that $|\cS(\gamma)|\le k$, and let $|\cS(\alpha)|=k+1$. Choose any $l\in \cS(\alpha)$, $l\neq m$. Then both orbits $\OO$ and $\OO_m$ are birationally induced from orbits $\OO'$ and $\OO'_m$ in $\fg\fl_l^{d_l}\times \fg_{n-2ld_l}$. If $\alpha'$ is the partition corresponding to $\OO'$, we have $\cS(\alpha')=\cS(\alpha)\setminus\{l\}$. We proceed analogously to 2) to show that $\fP_m=\fP'_m$.  By the induction step the latter can be identified with $\fz(\fg\fl_m^{k_m})$, where $k_m=d_m$ or $d_{m}+1$ depending on the partition $\alpha'$. Since both identifications $\fP_m=\fP'_m$ and $\widetilde{\fP}_m=\widetilde{\fP}'_m$ hold, we conclude that the symmetries coincide. We finish with observing that $\alpha_m$ and $\alpha_{m+1}$ are the only two odd members of $\alpha$, both appearing with multiplicity $1$ if and only if the same is true for $\alpha'_m$ and $\alpha'_{m+1}$.
	\end{proofl}

        We note that during the proof we obtain an explicit description of all codimension $2$ singularities together with symmetries. \cref{theorem_symmetry} is a direct corollary of results in this section.
   
	\subsection{Partial Namikawa spaces for $G$-equivariant covers}\label{Pcover}
	In this section we generalize the results of \cref{Porbit} for $G$-equivariant covers of nilpotent orbits. Let us fix some notations.
	
	Let $\widehat{\OO}$ be a $G$-equivariant cover of a nilpotent orbit $\OO\subset \fg$, and $\alpha\in \cP(n)$ be the partition corresponding to $\OO$. Let $\mu: \Spec(\CC[\widehat{\OO}])\to \Spec(\CC[\OO])$ be the map associated with the covering map $\widehat{\OO}\to \OO$, and consider $m\in \cS(\alpha)$. In this section we use the notation $\OO_m$ for the orbit in $\Spec(\CC[\OO])$ instead of the orbit in $\overline{\OO}$. That plays a role when $\overline{\OO}$ has a dimension $2$ singularity of type $A_{2k-1}\cup A_{2k-1}$, i.e. for the pair $(\OO, \OO_m)$ with the minimal degenerations of type (e). Under the normalization it becomes two copies of the singularity of type $A_{2k-1}$, the computation of partial Namikawa spaces done in the previous section shows that the singular point of both copies lie on the same symplectic leaf. Similar to above, we set $\OO_m=\widehat{\OO}_m^I\sqcup \widehat{\OO}_m^{II}$ for $\OO_m$ corresponding to a very even partition. We remark that if $\OO_m$ corresponds to a very even partition, $\alpha$ is a partition of type $D$ with only two odd members. \cref{pi_1} implies that $\pi_1(\OO)=1$, and thus it is already studied in the previous section. For the remainder of this section we assume that $\OO_m$ does not correspond to a very even partition, and is therefore a single orbit. 
 
    Since we are interested in the singularities of the normal scheme $\Spec(\CC[\widehat{\OO}])$, it is more convenient for us to take $\OO_m$ to be the codimension $2$ orbit for the trivial cover of $\OO$. We say that the cover $\widehat{\OO}_m=\mu^{-1}(\OO_m)$ of $\OO_m$ is the cover associated to the cover $\widehat{\OO}\to \OO$. If $\widehat{\OO}_m\not\subset \Spec(\CC[\widehat{\OO}])^{reg}$, every connected component $\fL_m^i$ of $\widehat{\OO}_m$ is a symplectic leaf of codimension $2$ in $\Spec(\CC[\widehat{\OO}])$, and we can assign to it the corresponding Namikawa space $\fP_m^i$. We set $\fP_m=\bigoplus \fP_m^i$, and call it the partial Namikawa space corresponding to $\widehat{\OO}_m$. If $\widehat{\OO}_m\subset \Spec(\CC[\widehat{\OO}])^{reg}$ we define $\fP_m=0$.
	
	 The following is the main result of this section.
	\begin{theorem}\label{4.4}
		Let $\widehat{\OO}$ be a cover of a nilpotent orbit $\OO\subset \fg$, and $\alpha\in \cP(n)$ be the partition corresponding to $\OO$. Consider $m\in \cS(\alpha)$, and let $\widehat{\OO}_m\subset \Spec(\CC[\widehat{\OO}])$ be the associated cover of $\OO_m$.
		
			Let $\widehat{\OO}'$ be a birationally rigid cover of $\OO'\subset \fl$, such that $\widehat{\OO}$ is birationally induced from $\widehat{\OO}'$. Suppose that 
			\begin{equation*}
				\fl=\prod_{m\in \cS(\alpha)}\fg\fl_m^{k_m}\times \prod_{k\in S} \fg\fl_k\times \fg_{n'}
			\end{equation*} for some set $S$, such that $S\cap \cS(\alpha)=\emptyset$. Then under the identification $\fP=\fz(\fl)$ of \cref{P=z(l)}
			\begin{itemize}
				\item The partial Namikawa space $\fP_m$ corresponding to $\widehat{\OO}_m$ is identified with $\fz(\fg\fl_m^{k_m})$;
				\item The smooth part of Namikawa space $\fP_0=H^2(\widehat{\OO}, \CC)$ is identified with $\fz(\prod_{k\in S} \fg\fl_k)$.
			\end{itemize} 
	\end{theorem} 
	\begin{rmk}
		The numbers $k_m$ and the set $S$ depend on the cover $\widehat{\OO}$ of $\OO$. The latter can be understood on a level of partitions as a set that controls all colapses on the way from the partition $\alpha'$ of $\OO'$ to the partition $\alpha$. Each $\fg\fl_l$ for $l\in S$ occurs only once in $\fl$.
	\end{rmk}
	Let us explain how we are going to prove the theorem. First, we need to show that the birational induction from the Levi algebra $\fl=\fg\fl_k\times \fg_{n-2k}$ does not change the singularity of number $m$ for $m\neq k$. This is done in \cref{lnotm} using and generalizing \cref{singind}. In addition, that allows us to control the corresponding Namikawa space $\fP_m$. \cref{h2} allows us to control the smooth part $\fP_0$ of the Namikawa space. Therefore, we can describe the change of the Namikawa space $\fP$ under the birational induction from $\fl=\fg\fl_k\times \fg_{n-2k}$, and thus get the result of the theorem.
	
	First, we need to consider a case when $\widehat{\OO}$ is birationally induced from an orbit in $\fl=\fg\fl_k\times \fg_{n-2k}$ rather than a cover of it.

%
	
	\begin{lemma}\label{together}
		Let $\OO'\subset \fg_{n-2k}$ be a nilpotent orbit, and $\OO'_m\subset \Spec(\CC[\OO'])$ be the codimension $2$ orbit with the number of singularity $m=q(\OO', \OO'_m)\neq k$. Let $\widehat{\OO}$, $\widehat{\OO}_m$ be the covers of orbits in $\fg_{n}$, such that $\widehat{\OO}$, $\widehat{\OO}_m$ are birationally induced from the Levi subalgebra $\fl=\fg\fl_k\times \fg_{n-2k}$ and orbits $\{0\}\times {\OO}'$, $\{0\}\times {\OO}'_m$ respectively. Let $\mu: \Spec(\CC[\widehat{\OO}])\to \Spec(\CC[\OO])$ be the map associated to the covering map $\widehat{\OO}\to \OO$. Then:
		\begin{itemize}
			\item[(i)] We have $\widehat{\OO}_m\simeq \mu^{-1}(\OO_m)\subset \Spec(\CC[\widehat{\OO}])$.
                \item[(ii)] The map $\mu$ is \'etale over $\OO\cup \OO_m$. 
			\item[(iii)] The singularities of $\widehat{\OO}_m\subset \Spec(\CC[\widehat{\OO}])$ and $\OO'_m\subset \Spec(\CC[\OO'])$ are equivalent to each other.
		\end{itemize} 
	\end{lemma}
	\begin{proof}
        
        Let us fix notations $X'=\Spec(\CC[{\OO}'])$, $X=\Spec(\CC[\OO])$, $\widehat{X}=\Spec(\CC[\widehat{\OO}])$, and pick a $\QQ$-factorial terminaliziation $\widetilde{X}'$.
        
        If $\OO$ is birationally induced from $\OO'$, then $\OO_m\subset \Spec(\CC[\OO])$, and we are done by \cref{singind}. Thus, by \cref{o2} we can assume that $\widehat{\OO}$ is a $2$-fold cover of $\OO$.
        
	   We claim that $\widehat{\OO}_m\to \OO_m$ is a $2$-fold cover. That together with (i) implies (ii). Let $\alpha$, $\beta$, $\alpha'$, and $\beta'$ be the partitions corresponding to $\OO$, $\OO_m$, $\OO'$, and $\OO'_m$ respectively. By \cref{o2} and \cref{induction}, $\widehat{\OO}\to \OO$ is not birational if and only if $\alpha$ is a collapse of $\alpha^{'k}$, and either $\alpha'$ is not very even, or $l$ is odd. 

        First, consider the case when $\beta'$ is very even. As shown in the proof of \cref{singind}, $(\alpha', \beta')$ is of type (a). It follows that $\OO_m$ is a trivial cover of the nilpotent orbit corresponding to the partition $\beta$. Since $k\neq m$, $\alpha$ is a collapse of $\alpha^{'k}$ if and only if $\beta$ is a collapse of $\beta^{'k}$. We have $\OO'_m=\OO_m^{'I} \sqcup \OO_m^{'II}$, and therefore $\widehat{\OO}_m$ is a union of two disjoint copies of $\OO_m$. 
  
	    The image of the generalized Springer morphism $G\times^P(X'\times \fn)\to \fg$ is $\overline{\OO}$, so the map factors through $\mu': G\times^P(X')\times \fn)\to X$. Recall from \cref{P=z(l)} that there is a map $\rho: G\times^P(X'\times \fn)\to \widehat{X}$. We want to show that $\mu^{-1}(\OO_m)=\widehat{\OO}_m$. We have the following commutative diagram.
		\begin{equation*}\label{co}
		\xymatrix{
			G\times^P(\widetilde{X}'\times \fn) \ar[r] &G\times^P(X'\times \fn) \ar[dr]_{\mu'} \ar[rr]^{\rho}&& \widehat{X}\ar[dl]^{\mu}& \\
			&& X&&}
		\end{equation*}
		
		Let $\Sigma_X^{\wedge_0}\simeq (\CC^2/\Gamma)^{\wedge_0}$ be the formal slice to $\OO_m$ in $X$, $V$ be the formal neighborhood of a point $x$ in $\OO_m$, and $U=V\times \Sigma_X^{\wedge_0}$ be a formal neighborhood of a point $x$ in $X$. Let $\Sigma_X=\CC^2/\Gamma_X$, and $\widetilde{\Sigma}$ be its minimal resolution. Similarly, considering the formal slice to $\OO'_m$ in $X'$ we define $\Gamma_{X'}$, $\Sigma_{X'}$, and $\widetilde{\Sigma}_{X'}$. In what follows we say that $\Sigma_X$ is the singularity of $\OO_m$ in $X$.
		
		First, suppose that $\OO_m\xrightarrow{\sim}\mu^{-1}(\OO_m)\subset \widehat{X}$. Then the formal slice to $\OO_m$ in $\widehat{X}$ is $\Sigma_{\widehat{X}}^{\wedge_0}$, where $\Sigma_{\widehat{X}}=\CC^2/\Gamma_{\widehat{X}}$ for some subgroup $\Gamma_{\widehat{X}}\subset \Gamma$ of index $2$, and $\mu^{-1}(U)\simeq V\times \Sigma_{\widehat{X}}^{\wedge_0}$. Let $\widetilde{\Sigma}_{\widehat{X}}$ be the minimal resolution of $\Sigma_{\widehat{X}}$.
		
		The preimage $\mu'^{-1}(\OO_m)\subset G\times^P (X'\times \fn)$ contains the open $G$-orbit in $G\times^P (\Spec(\CC[\OO'_m])\times \fn)$, that is $\widehat{\OO}_m$. Therefore, $\mu'^{-1}(U)$ contains a disjoint union of formal neighborhoods to $x_1$, $x_2\in \mu'^{-1}(x)\cap \widehat{\OO}_m$. By the argument of \cref{monodromy}, the formal slices at points $x_1$, $x_2$ are $\Sigma_{X'}^{\wedge_0}$. 

        The composition $G\times^P(\widetilde{X}'\times \fn) \to \widehat{X}$ is a $\QQ$-factorial terminalization, and therefore it gives a minimal resolution of $\Sigma_{\widehat{X}}$. Thus, $\rho^{-1}(\Sigma_{\widehat{X}})$ is a partial resolution of a Kleinian singularity $\Sigma_{\widehat{X}}$. It follows that $2\dim H^2(\widetilde{\Sigma}_{X'})\le \dim H^2(\widetilde{\Sigma}_{\widehat{X}})$. It remains to check that this is not possible. 
        
        Assume first that the minimal degenerations of $(\alpha, \beta)$ and $(\alpha', \beta')$ coincide. If $\Sigma_X=\Sigma_{X'}$ is not of type $D$, then $\dim H^2(\widetilde{\Sigma}_{\widehat{X}})< \dim H^2(\widetilde{\Sigma}_X)$, and therefore $2\dim H^2(\widetilde{\Sigma}_{X'})< \dim H^2(\widetilde{\Sigma}_X)$, we get a contradiction. Assume that $\Gamma$ is a dihedral group on $4t$ elements, and $\Gamma_{\widehat{X}}$ is a cyclic subgroup of order $2t$. Then $\dim H^2(\widetilde{\Sigma}_{\widehat{X}})=2t-1$, and $\dim H^2(\widetilde{\Sigma}_X)=t+2$. Thus, $2t+4\le 2t-1$, we get a contradiction. 

        Arguing as in the proof of \cref{singind}, we may assume that $m-1\le k\le m+2$, and we are in one of the following cases. 

         \begin{itemize}
            \item[(i)] $k=m-1$, and $(\alpha', \beta')$ is of type (d);
            \item[(ii)] $k=m-1$, and $(\alpha', \beta')$ is of type (e);
            \item[(iii)] $k=m+1$, and $(\alpha', \beta')$ is of type (c);
            \item[(iv)] $k=m+1$, and $(\alpha', \beta')$ is of type (e).
        \end{itemize}

        In all $4$ cases we have $\dim H^2(\widetilde{\Sigma}_{X'})=2t-1$, where $t=\left[\frac{\alpha'_m-\alpha'_{m+1}}{2}\right]$. A direct computation for each of these cases shows that

        \begin{itemize}
            \item[(i)]$(\alpha, \beta)$ is of type (b). $\Sigma_X$ is of type $D_{t+2}$;
            \item[(ii)] $(\alpha, \beta)$ is of type (c). $\Sigma_X$ is of type $A_{2t-1}$;
            \item[(iii)] $(\alpha, \beta)$ is of type (b). $\Sigma_X$ is of type $D_{t+2}$;
            \item[(iv)] $(\alpha, \beta)$ is of type (d). $\Sigma_X$ is of type $A_{2t-1}$.
        \end{itemize}

            In all cases $\dim H^2(\widetilde{\Sigma}_{\widehat{X}})\le 2t-1$ for all possible $\Gamma_{\widehat{X}}\subset \Gamma$, and we get a contradiction.
	\end{proof}
	
		We proceed to obtain the general case from the lemma.
	
	\begin{prop}\label{lnotm}\label{le}
		Let $k$, $\fl$, $\OO'$, $\OO'_m$ be as in \cref{together}, $\widehat{\OO}'$ be a $G_{n-2k}$-equivariant cover of $\OO'$, and $\mu': \Spec(\CC[\widehat{\OO}'])\to \Spec(\CC[\OO'])$ be the natural map associated with the cover. Let $\widehat{\OO}'_m=\mu'^{-1}(\OO_m)\subset \Spec(\CC[\widehat{\OO}'])$ be the corresponding cover of $\OO'_m$ (possibly disconnected). Let $\widehat{\OO}$, $\widehat{\OO}_m$ be covers birationally induced from $\fl$ and covers $\{0\}\times \widehat{\OO}'$, $\{0\}\times \widehat{\OO}'_m$. Let $\mu: \Spec(\CC[\widehat{\OO}])\to \Spec(\CC[\OO])$ be the natural covering map. Then:
		\begin{itemize}
			\item[(i)]  $\widehat{\OO}_m\xrightarrow{\sim}\mu^{-1}(\OO_m)\subset \Spec(\CC[\widehat{\OO}])$;
			\item[(ii)]  The singularity of every connected component of $\widehat{\OO}_m$ in $\Spec(\CC[\widehat{\OO}])$ is equivalent to the singularity of every connected component of $\widehat{\OO}'_m$ in $\Spec(\CC[\widehat{\OO}'])$;
			\item[(iii)]  Let $\fP_m$, $\fP'_m$ be the partial Namikawa spaces corresponding to $\widehat{\OO}_m$ and $\widehat{\OO}'_m$ respectively. (If $\widehat{\OO}_m\subset \Spec(\CC[\widehat{\OO}])^{reg} $, we set $\fP_m=0$.) Then $\fP_m\subset \fP'_m$. 
		\end{itemize} 
	\end{prop}
	\begin{proof}
		To show that $\widehat{\OO}_m\subset \Spec(\CC[\widehat{\OO}])$ we use the induction on the degree of the cover $ \widehat{\OO}'\to \OO'$. The cover of degree $1$ is considered in \cref{together}. 
		
		Let $\check{\OO}'$ be an $G_{n-2k}$-equivariant cover of $\OO'$, and $\check{\OO}'_m$ be the associated cover of $\OO'_m$. Let $\check{\OO}$ and $\check{\OO}_m$ be covers of orbits $\OO$ and $\OO_m$ that are birationally induced from $\check{\OO}'$ and $\check{\OO}'_m$ respectively. Let $\check{X}=\Spec(\CC[\check{\OO}])$, $\check{X}'=\Spec(\CC[\check{\OO}'])$, and let $\check{\mu}$ stand for the natural map $\check{X}\to X=\Spec(\CC[\OO])$. Moreover, suppose that $\check{\OO}_m\xrightarrow{\sim} \check{\mu}^{-1}(\OO_m)\subset \check{X}$. 
		
		Let $\widehat{\OO}'\to \check{\OO}'$ be a $G$-equivariant $2$-fold cover. Let $\widehat{X}=\Spec(\CC[\widehat{\OO}])$, $\widehat{X}'=\Spec(\CC[\widehat{\OO}'])$, and let $\widehat{\OO}'_m\subset \widehat{X}$ be the associated cover of $\OO'_m$. We will prove that $\widehat{\OO}_m\xrightarrow{\sim} \mu^{-1}(\OO_m)\subset \widehat{X}$. That would complete the induction step for (i).
		
		We have the following commutative diagram. The maps $f$ and $f'$ are naturally induced from the covers $\widehat{\OO}\to \check{\OO}$ and $\widehat{\OO}'\to \check{\OO}'$ respectively. The maps $\widehat{\rho}$ and $\check{\rho}$ are constructed analogously to $\rho$ in \cref{P=z(l)}.
		\begin{equation*}\label{coco}
		\xymatrix{
			&G\times^P(\widehat{X}'\times \fn) \ar[d]^{\widehat{\rho}} \ar[r]^{f'}& G\times^P(\check{X}'\times \fn) \ar[d]^{\check{\rho}} & \\
			&\widehat{X}\ar[r]^{f}& \check{X}&}
		\end{equation*}
		
		We will need the following lemma:
		\begin{lemma}\label{sur}
			The morphisms $\widehat{\rho}$ and $f$ are surjective. 
		\end{lemma}
		\begin{proof}
			The map $\widehat{\rho}$ is by construction a partial resolution of singularities. A birational proper map is surjective.
			
			The map $f$ is the morphism of affine scheme induced from the integral extension $\CC[\check{\OO}]\to \CC[\widehat{\OO}]$. The Going Up theorem implies the surjectivity.
		\end{proof}
		
		Let us consider two cases:
		
		1) $\widehat{\OO}'_m=\check{\OO}'_m$. Therefore $\widehat{\OO}_m=\check{\OO}_m$. Let $\mathring{\OO}_m\subset\widehat{X}$ be the corresponding cover of $\OO_m$. Map $f$ is $G$-equivariant by construction and surjective by \cref{sur}, and therefore induces the cover $\mathring{\OO}_m\to \check{\OO}_m$. We have $\widehat{\OO}_m=f'^{-1}\check{\rho}^{-1}(\check{\OO}_m)$. Thus, $\widehat{\OO}_m=\widehat{\rho}^{-1}f^{-1}(\check{\OO}_m)=\widehat{\rho}^{-1}(\mathring{\OO}_m)$.
		
		Since $\widehat{\rho}$ is surjective, it induces a cover $\widehat{\OO}_m\to \mathring{\OO}_m$. It follows that both covers are isomorphisms.
		
		2) $\widehat{\OO}'_m$ is a $2$-fold cover of $\check{\OO}'_m$. Let $p=\check{\rho}\circ f'$. The map $p$ restricted to $\widehat{\OO}_m$ is a $2$-fold cover. If $f^{-1}(\check{\OO}_m)$ is a $2$-fold cover of $\check{\OO}_m$, then $\widehat{\rho}$ gives an isomorphism $f^{-1}(\check{\OO}_m)\simeq \widehat{\OO}_m$. Assume that $f^{-1}(\check{\OO}_m)\simeq \check{\OO}_m$. We can proceed similarly to \cref{together}. 

        Let $\Sigma_{\check{X}}=\CC^2/\Gamma_{\check{X}}$ be the singularity of $\check{\OO}_m$ in $\check{X}$. By the induction hypothesis, the singularity of $\check{\OO}'_m$ in $\check{X}'$ is isomorphic to $\Sigma_{\check{X}}$. The singularity of $f^{-1}(\check{\OO}_m)$ in $\widehat{X}$ is $\Sigma_{\widehat{X}}=\CC^2/\Gamma_{\widehat{X}}$ for a subgroup $\Gamma_{\widehat{X}}\subset \Gamma_{\check{X}}$ of index $2$. Let $\Sigma_{\widehat{X}'}=\CC^2/\Gamma_{\widehat{X}'}$ be the singularity of $\widehat{\OO}'_m$ in $\widehat{X}'$, and $\Gamma_{\widehat{X}'}\subset \Gamma_{\check{X}}$ is a subgroup of index $2$. As in \cref{together}, if $\widehat{\OO}_m\neq f^{-1}(\check{\OO}_m)$ then $2\dim H^2(\widetilde{\Sigma}_{\widehat{X}'})\le \dim H^2(\widetilde{\Sigma}_{\widehat{X}})$. However, since both $\Gamma_{\widehat{X}}$ and $\Gamma_{\widehat{X}'}$ are subgroups of index $2$ in $\Gamma_{\check{X}}$, and $\Sigma_{\widehat{X}'}$ appears in the partial resolution of $\Sigma_{\widehat{X}}$, it follows that $\Sigma_{\widehat{X}}\simeq \Sigma_{\widehat{X}'}$. We got a contradiction. (i) and (ii) follow.

		It remains to prove that for the corresponding partial Namikawa spaces $\fP_m$ and $\fP'_m$ we have $ \fP_m\subset \fP'_m$. Since the singularity of $\widehat{\OO}_m$ in $ \widehat{X}$ is equivalent to the one of $\widehat{\OO}'_m$ in $ \widehat{X}'$, we have $\widetilde{\fP}_m=\widetilde{\fP}'_m$. If the action of $\pi_1(\widehat{\OO}'_m)$ on $\widetilde{\fP}'_m$ is nontrivial, then \cref{monodromy} implies that $\fP_m=\fP'_m$. If $\fP'_m=\widetilde{\fP}'_m$, the statement is obvious.  
		%
	\end{proof}
	Note that by construction $\widehat{\OO}_m$ has the same number of connected components as $\widehat{\OO}'_m$. By applying \cref{le} multiple times we get the following corollary.
 
	\begin{cor}\label{lem}
		Suppose that $\widehat{\OO}$ is birationally induced from the birationally rigid cover $\widehat{\OO}_L$ of an orbit $\OO_L\subset \fl=\prod \fg\fl_i^{k_i}\times \fg_{n'}$, and let $\alpha$ be the partition corresponding to the orbit $\OO$. Then for any $m\in S(\alpha)$ we have $\fP_m\subset \fz(\fg\fl_m^{k_m})$. 
	\end{cor} 
	\begin{proof}
		Since $\widehat{\OO}_L$ is birationally rigid, we have $\OO_L\subset \fg_{n'}$. Let $\widehat{\OO}'$ be a cover of an orbit in $\fg_{n'+2mk_m}$ that is birationally induced from $\fg\fl_m^{k_m}\times \fg_{n'}$. Consider $\OO'_m\subset \overline{\OO}'$ to be the orbit with the $n(\OO', \OO'_m)=m$, and let $\widehat{\OO}'_m\subset \widehat{X}'=\Spec(\CC[\widehat{\OO}'])$ be the corresponding cover of $\OO'_m$. Let $\fP(\widehat{X}')$ and $\fP_m (\widehat{X}')$ be the Namikawa space of $\widehat{X}'$ and the (direct sum of) partial Namikawa space(s) corresponding to (connected components of) $\widehat{\OO}'_m$ respectively. By \cref{le} we have $ \fP_m\subset \fP_m (\widehat{X}')$. Also, we have $\fP_m (\widehat{X}')\subset \fP (\widehat{X}')$, and by \cref{P=z(l)} $\fP (\widehat{X}')=\fz(\fg\fl_m^{k_m}\times \fg_{n'})=\fz(\fg\fl_m^{k_m})$. Combining, we get the statement of the corollary.
	\end{proof}
	\begin{prop}\label{h2}
		Suppose that $\widehat{\OO}$ is birationally induced from a birationally rigid cover $\widehat{\OO}_L$, where $\widehat{\OO}_L$ is a cover of $\OO_L\subset \fl$. Let $\alpha\in \cP(n)$ be the partition corresponding to the orbit $\OO$. Recall from \cref{lll} that $\fl$ is $G$-conjugate to $\prod_{i\in \cS(\alpha)} \fg\fl_i^{k_i}\times \prod_{i\in S} \fg\fl_i\times \fg_{t}$, where $S$ is some finite set, $S\cap \cS(\alpha)=\emptyset$ and $t=n-2\sum_{i\in \cS(\alpha)} ik_i-2\sum_{i\in S} i$. Then $\fP_0=H^2(\widehat{\OO}, \CC)\subset \fP$ is identified with $\CC^{|S|}=\fz(\prod_{i\in S} \fg\fl_i)\subset \fz(\fl)$. 
	\end{prop}
	\begin{proof}
		Since $\widehat{\OO}_L$ is birationally rigid for the Levi $\fl=\prod_{i\in \cS(\alpha)} \fg\fl_i^{k_i}\times \prod_{i\in S} \fg\fl_i\times \fg_{t}$, $\widehat{\OO}_L$ is a cover of an orbit in $\fg_t$.
		
		Recall from \cref{LS} that the orbit $\OO$ can be induced from $\{0\}\times \OO'\subset \fg\fl_k\times \fg_{n-2k}$ for $k\notin \cS(\alpha)$ if and only if $\alpha$ is the collapse of $\alpha'^k$, where $\alpha'$ is the partition corresponding to $\OO'$. Then $\alpha_{k-1}>\alpha_k=\alpha_{k+1}>\alpha_{k+2}$, and $\alpha_k$ has the opposite parity to $\fg$. Let $\check{\OO}$ be the cover of $\OO$ birationally induced from $\OO'$. By \cref{bircov} a cover $\widehat{\OO}$ is birationally induced from some cover $\widehat{\OO}'$ of $\OO'$ if and only if $\pi_1^G(\widehat{\OO})\subset \pi_1^G(\check{\OO})$. Let $X=\Spec(\CC[\widehat{\OO}])$, and $X'=\Spec(\CC[\widehat{\OO}'])$. Let $\fP (X)$ and $\fP (X')$ be the Namikawa spaces for $X$ and $X'$ respectively. 
		
		Recall from the proof of \cref{Q-term} that $H^2(X^{reg}, \CC)=H^2(\widehat{\OO},\CC)$. Analogously to \cref{o2}, $\dim H^2(\widehat{\OO}, \CC)=\dim H^2(\widehat{\OO}',\CC)+1$. Let Levi subalgebra $\fl'\subset \fg_{n-2k}$ and a birationally rigid cover $\widehat{\OO}_{L'}$ of an orbit $\OO_{L'}\subset \fl'$ be such that $\widehat{\OO}'$ is birationally induced from $\widehat{\OO}_L$. Then $\widehat{\OO}$ is birationally induced from $\{0\}\times \widehat{\OO}_{L'}$ with the Levi subalgebra $\fg\fl_k\times \fl'$, and we have $\fP (X)=\fz(\fg\fl_k\times \fl)=\fz(\fg\fl_k)\oplus \fP' (X')$. Recall that by \cref{le} $\dim \fP_m (X)\le \dim \fP_m (X')$ for all $m\in \cS(\alpha)$. 		
		Therefore $\dim \fP_m (X)=\dim \fP_m (X')$ for all $m\in \cS(\alpha)$, and $H^2(\widehat{\OO},\CC)=H^2(\widehat{\OO}',\CC)\oplus\fz(\fg\fl_k)$.
		
		Set $r=n-2\sum_{i\in S} i$, and let $\widehat{\OO}_1$ be the cover of an orbit ${\OO}_1\subset \fg_{r}$ that is birationally induced from $\{0\}\times\ldots\times \{0\}\times \widehat{\OO}_L\subset \prod_{i\in \cS(\alpha)} \fg\fl_i^{k_i}\times \fg_t$. Then $H^2(\widehat{\OO},\CC)=H^2(\widehat{\OO}_1,\CC)\oplus\bigoplus_{i\in S}\fz(\fg\fl_i)$. Therefore, it is enough to show that $H^2(\widehat{\OO}_1,\CC)=0$. 
		
		Let $\beta$ be the partition corresponding to $\OO_1\subset \fg_r$. Note that by \cref{restrict} $\widehat{\OO}_1$ cannot be birationally induced from any Levi subalgebra of form $\fg\fl_k\times \fg_{r-2k}$ for $k\not\in \cS(\beta)$.
		
		Suppose that $H^2(\widehat{\OO}_1, \CC)\neq 0$. Recall from \cref{coh} that $H^2(\widehat{\OO}_1, \CC)=(\fz^*)^{\pi_1(\widehat{\OO}_1)}$, where $\fz$ is the center of the Lie algebra of the reductive part $\QQ$ of the stabilizer of an element $x\in \widehat{\OO}_1$. Since $\fq$ is the sum of some copies of $\fs\fo_a$ and $\fs\fp_b$, $\fz$ is a sum of copies of $\fs\fo_2$ corresponding to the pairs ($m$, $k$), such that $m$ and $\fg$ are of the opposite parity, and $\beta_{k-1}>\beta_k=m=\beta_{k+1}>\beta_{k+2}$. We denote this copy of $\fs\fo_2$ by $\fs\fo_{2,k}$. Suppose that $\fs\fo_{2,k}\subset H^2(\widehat{\OO}_1, \CC)$. Then $\pi_1^G(\widehat{\OO}_1)$ acts on the $\fs\fo_{2,k}$ trivially, so $\pi_1^G(\widehat{\OO}_1)\subset \Gamma_k$, where $\Gamma_k$ is defined as in \cref{birLS}. By \cref{bircov}, $\widehat{\OO}_1$ is birationally induced from a cover of a nilpotent orbit in $\fg\fl_k\times \fg_{r-2k}$. But such $k\notin \cS(\beta)$, so we get a contradiction. Therefore $H^2(\widehat{\OO}_1, \CC)=0$.
	\end{proof}
	
	Note that \cref{h2} is the second statement of \cref{4.4}. To finish the proof of the theorem we imply the first statement from \cref{h2} and \cref{lem}. 
	\begin{cor}\label{contractsing}
		Let $\widehat{\OO}$ be a cover of a nilpotent orbit $\OO\subset \fg$, and $\alpha\in \cP(n)$ be the partition corresponding to $\OO$. Consider $m\in \cS(\alpha)$ and let $\widehat{\OO}_m\subset \Spec(\CC[\widehat{\OO}])$ be the corresponding cover of $\OO_m$. Let $\widehat{\OO}'$ be a birationally rigid cover of $\OO'\subset \fl$, such that $\widehat{\OO}$ is birationally induced from $\widehat{\OO}'$. Suppose that $\fl=\prod_{m\in \cS(\alpha)}\fg\fl_m^{k_m}\times \prod_{k\in S} \fg\fl_k\times \fg_{n'}$. Then $\fP_m$ is identified with $\CC^{k_m}=\fz(\fg\fl_m^{k_m})$.
	\end{cor} 
	\begin{proof}
		By \cref{h2}, $\fP_0=H^2(\widehat{\OO},\CC)=\fz(\prod_{k\in S} \fg\fl_k)$. By \cref{lem}, $\fP_m\subset \fz(\fg\fl_m^{k_m})$. Recall that $\fP=\bigoplus_{m\in \cS(\alpha)}\fz(\fg\fl_m^{k_m}) \oplus \bigoplus_{l\in S}\fz(\fg\fl_l)$  by \cref{P=z(l)}, and $\fP=\fP_0\oplus \bigoplus_{m\in \cS(\alpha)} \fP_m$ by \cref{P=P}. Therefore, we have $\fP_m=\fz(\fg\fl_m^{k_m})$ for all $m\in S(\alpha)$.
	\end{proof}

%
%
%
%


	\section{Singularities of $\Spec(\CC[\widetilde{\OO}])$}\label{affin}

        We note that the arguments of \cref{sect: partial} work for every $G$-equivariant cover. However, the dimension of the partial Namikawa space $\fP_m$ in general depends on the cover $\widehat{\OO}$ of $\OO$. To give a general answer, in this section we compute the partial Namikawa spaces and the corresponding singularities for the universal cover $\widetilde{\OO}$. For a general cover $\widehat{\OO}$, the singularity of the associated cover $\widehat{\OO}_m$ is a cover of the singularity of $\OO_m\subset \Spec(\CC[\OO])$ and is covered by the singularity of $\widetilde{\OO}_m\subset \Spec(\CC[\widetilde{\OO}])$. Using \cref{4.4}, we can explicitly compute it.  
	
	\subsection{Computing the Namikawa space $\fP_m$}\label{compute p}
	Let ${\OO}\subset \fg$ be a nilpotent orbit, $\alpha\in \cP(n)$ be the corresponding partition, and consider $m\in \cS(\alpha)$. Let $\beta$ be the partition corresponding to the orbit $\OO_m$. Recall that $d_m=[\frac{\alpha_m-\alpha_{m+1}}{2}]$, and $\OO$ is birationally induced from $\{0\}\times \OO_L\subset \fg\fl_m^{d_m}\times \fg_{n-2md_m}$. Let $\gamma$ be the partition corresponding to the orbit $\OO_m$. By \cref{pi1g}, we have a group epimorphism $\phi: \pi_1^G(\OO)\to \pi_1^{G_{n-2md_m}}(\OO_L)$, and 
 	let $H_m\subset \pi_1^G(\OO)$ be the kernel of this map. 
	
	\begin{prop}\label{hm}
		Consider types of the singularity $\OO_m\subset \overline{\OO}$ as in \cref{KP}. Then we have the following correspondence between a type of the singularity and $H_m$:
		\begin{center}
			\begin{tabular}{ |c|c|} 
				\hline
				Type of the singularity $\OO_m\subset \overline{\OO}$ & $H_m$\\
				\hline 							 		
			a, $\alpha_m$ and $\alpha_{m+1}$ are the only members with odd multiplicity & $1$\\
			\hline 
			a, otherwise & $\ZZ/2\ZZ$\\
				\hline 
				b, $\alpha_m$ and $\alpha_{m+1}$ are the only members with odd multiplicity & $1$\\
				\hline 
				b, otherwise & $\ZZ/2\ZZ$\\
				\hline 
				c&  $1$\\
				\hline 
				d& $1$\\
				\hline 
				e& $1$\\
				\hline 
			\end{tabular}
		\end{center}
	\end{prop}
	\begin{proof}
%
%
	Follows from the direct computation of $|\pi_1^G(\OO)|$ and $|\pi_1^{G_{n-2md_m}}(\OO_L)|$ using \cref{pi_1}.
	\end{proof}
	%
	Now, let $\widetilde{\OO}$ be the universal $G$-equivariant cover of $\OO$, and let $\widetilde{\OO}_m\subset \Spec(\CC[\widetilde{\OO}])$ be the preimage of $\OO_m\subset \overline{\OO}$ under the moment map $\Spec(\CC[\widetilde{\OO}])\to \fg^*$. 
	
	\begin{prop}\label{pi}
            \begin{itemize}
                \item[(i)] If $\alpha_m$ and $\alpha_{m+1}$ both appear with multiplicity $1$, then $\dim \fP_m=d_m+1$;
                \item[(ii)] If $H_m=1$, but not as in (i), then $\dim \fP_m=d_m$;
                \item[(iii)] If $H_m=\ZZ/2\ZZ$, $\dim \fP_m=d_m-1$.
            \end{itemize}
	\end{prop}
	\begin{proof}
		If $H_m=1$, then $\pi_1^{G-2md_m}(\OO_L)=\pi_1^G(\OO)$, and therefore $\widetilde{\OO}$ is birationally induced from $\fg\fl_m^{d_m}\times \fg_{n-2md_m}$ by \cref{birbircov}. Moreover, the orbit $\OO_L$ is induced from $\fg\fl_m\times \fg_{n-2m(d_m+1)}$ if and only if $\alpha_m$ and $\alpha_{m+1}$ both appear with multiplicity $1$. In this case, the induction is not birational, and therefore by \cref{birbircov} and the discussion before, $\widetilde{\OO}_L$ is birationally induced from $\fg\fl_m\times \fg_{n-2m(d_m+1)}$. By \cref{contractsing}, $\fP_m=\fz(\fg\fl_m^{d_m+1})$ if $\alpha_m$ and $\alpha_{m+1}$ both appear with multiplicity $1$, and $\fP_m=\fz(\fg\fl_m^{d_m})$ otherwise.
		
		If $H_m=\ZZ/2\ZZ$, then \cref{birbircov} implies that $\widetilde{\OO}$ cannot be birationally induced from $\fg\fl_m^{d_m}\times \fg_{n-2md_m}$, so by \cref{contractsing} $\dim \fP_m<d_m$. If $d_m=1$, then $\dim \fP_m=0$. Suppose $d_m>1$. Let $\OO'\subset \fg\fl_m^{d_m-1}\times \fg_{n-2m(d_m-1)}$ be the orbit birationally induced from $\OO_L$. Since $\alpha_m>\alpha_{m+1}+2d_m-2$, $\pi_1^G(\OO)\simeq \pi_1^G(\OO')$, and $\OO$ is birationally induced from $\OO'$ by \cref{birLS}. By \cref{birbircov}, $\widetilde{\OO}$ is birationally induced from some $G_{n-2m(d_m-1)}$-equivariant cover $\widetilde{\OO}'$, and $\fz(\fg\fl_m^{d_m-1})\subset \fP_m$ by \cref{contractsing}. Therefore $\dim \fP_m\ge d_m-1$.
	\end{proof} 
	
	The cases (a, (i) -- (iii)) of \cref{theorem} immediately follow from \cref{pi}.

        To understand other cases we need a notion of an \emph{$m$-determining} cover $\widehat{\OO}$ of $\OO$. We define it as follows. Let $\fl=\prod_{m\in \cS(\alpha)}\fg\fl_m^{k_m}\times \prod_{k\in S} \fg\fl_k\times \fg_{n'}$, and assume that $(\fl ,\widetilde{\OO}_L)$ is a birational induction datum of $\widetilde{\OO}$. By \cref{pi}, $d_m-1\le k_m\le d_m+1$. Let $\fl'=\fg\fl_m^{k_m}\times \fg_{n-2md_m}$, and let $\OO_{L'}$, $\widetilde{\OO}_{L'}$ be the orbit induced from $\OO_L\subset \fl$, and a cover birationally induced from $\widetilde{\OO}_L$ respectively.
        
        \begin{itemize}
            \item[(i)] If $k_m\ge d_m$, we define $\widehat{\OO}$ to be the cover of $\OO$ birationally induced from the orbit $\OO_{L'}\subset \fl'$. We note that $\widehat{\OO}$ is either trivial or $2$-fold cover of $\OO$.
            \item[(ii)] If $k_m = d_m-1$, then $H_m\simeq \ZZ/2\ZZ$. Let $K_m\subset \pi_1^G(\OO)$ be a minimal subgroup such that the restriction of $\phi$ to $K_m$ is surjective. We define $\widehat{\OO}$ to be the cover corresponding to the subgroup $K_m\subset \pi_1^G(\OO)$, and note that it is a $2$-fold cover.
        \end{itemize}
        
         The key property is that by construction $\fP_m (\widetilde{X})\simeq \fz(\fg\fl_m^{k_m})\simeq \fP_m (\widehat{X})$. We will compute the singularity of $\widehat{\OO}_m$ in $ \Spec(\CC[\widehat{\OO}])$ in \cref{min}. In \cref{final} we prove that is it equivalent to the one of $\widetilde{\OO}_m$ in $\Spec(\CC[\widetilde{\OO}])$, thus justifying its name.

	
	
	\subsection{Singularity of the affinization of an $m$-determining cover}\label{min}
	
	Let us first compute the singularity of every connected component of $\widehat{\OO}_m$ in $\Spec(\CC[\widehat{\OO}])$ for the $m$-determining cover $\widehat{\OO}$. 
	If $\widehat{\OO}=\OO$, and we are done by \cref{KP}. Suppose $H_m=\ZZ/2\ZZ$. By \cref{hm} the singularity of the pair $(\OO, \OO_m)$ is of types (a) or (b). The type (a) was already discussed in \cref{compute p}, so we can assume that the singularity of $\OO_m\subset \overline{\OO}$ is equivalent to $\CC^2/\Gamma$, where $\Gamma$ stands for the binary dihedral group of order $4(d_m-1)$.
	
	\begin{lemma}\label{2cov}
		$\widehat{\OO}_m\simeq \OO_m$. Therefore $\widehat{\OO}_m$ is connected, and the singularity of $\widehat{\OO}_m$ in $\Spec(\CC[\widehat{\OO}])$ is equivalent to $\CC^2/\Gamma'$, where $\Gamma'\subset \Gamma$ is a subgroup of index $2$.
	\end{lemma}	
	\begin{proof}
		Let the singularity of every connected component of $\widehat{\OO}_m$ in $\Spec(\CC[\widehat{\OO}])$ be equivalent to $\CC^2/\Gamma_{\widehat{X}}$. Since $\widehat{\OO}\to \OO$ is a two-fold covering, $\Gamma_{\widehat{X}}\subset \Gamma$ is of index $1$ or $2$. 
		
		If $\Gamma_{\widehat{X}}=\Gamma$, then the singularity of every connected component $\widehat{\OO}_m$ in $\Spec(\CC[\widehat{\OO}])$ is equivalent to $\CC^2/\Gamma$. Therefore $\widetilde{\fP}_m (\widehat{X})=(\CC^{d_m+1})^{\oplus i}$, where $i=1$ or $2$ stands for the number of connected components of $\widehat{\OO}_m$, and $\dim \fP_m (\widehat{X})\ge d_m$. Analogously to \cref{pi}, $\dim \fP_m (\widehat{X})=d_m-1$, so we get a contradiction. 
		
		Thus, $\Gamma_{\widehat{X}}\subset \Gamma$ is a subgroup of index $2$, and therefore $\widehat{\OO}_m\simeq \OO_m$.
		
	\end{proof}	
	
	\begin{prop}\label{noD}
		$\Gamma_{\widehat{X}}$ is a cyclic group of order ${2(d_m-1)}$. Therefore the singularity of $\widehat{\OO}_m$ in $\Spec(\CC[\widehat{\OO}])$ is of type $B_{d_m-1}$.
	\end{prop}
	\begin{proof}
		Let $\fP_m$ and ${\fP}_m (\widehat{X})$ be the Namikawa spaces for the codimension two leaves $\OO_m\subset \Spec(\CC[{\OO}])$ and $\widehat{\OO}_m\subset \Spec(\CC[\widehat{\OO}])$ respectively. Note that by \cref{2cov} $\widehat{\OO}_m$ is connected. Let $\Sigma=\CC^2/\Gamma$ and $\Sigma_{\widehat{X}}=\CC^2/\Gamma_{\widehat{X}}$. 

        Let $V=Z(\CC[\Gamma])$ and $V_{\widehat{X}}=Z(\CC[\Gamma_{\widehat{X}}])$. Recall that $\fP_m\simeq V^{\pi_1(\OO_m)}$ and $\fP_m (\widehat{X}) \simeq V_{\widehat{X}}^{\pi_1(\OO_m)}$ respectively. 
		
		
		By the discussion above, ${\fP}_m (\widehat{X})=\fz(\fg\fl^{d_m-1})\subset \fz(\fg\fl^{d_m})=\fP_m$. 
  It implies that any $\xi\in \fP_m (\widehat{X})$ is a common deformation direction of $\CC^2/\Gamma$ and $\CC^2/\Gamma_{\widehat{X}}$, i.e. $\xi \in V\cap {V}_{\widehat{X}}$. Suppose that $\Gamma_{\widehat{X}}$ is not a cyclic group, then $d_m=2k+1$, $k\ge 1$, and $\Gamma_{\widehat{X}}$ is the dihedral group of order $4k$. Then $V\cap {V}_{\widehat{X}}=Z(\CC[\Gamma])\cap Z(\CC[\Gamma_{\widehat{X}}])=Z(\CC[\Gamma_{\widehat{X}}])^{\ZZ/2\ZZ}$, where $\ZZ/2\ZZ$ corresponds to the quotient $\Gamma/\Gamma_{\widehat{X}}$. It is easy to see that the action of $\ZZ/2\ZZ$ on $Z(\CC[\Gamma_{\widehat{X}}])$ is not trivial. Then we have $\dim \fP_m (\widehat{X})<\dim H^2(\widetilde{\Sigma}_{\widehat{X}})=\dim V_{\widehat{X}}\le k+2\le d_m-1$, and we get a contradiction.

  Thus, $\Gamma_{\widehat{X}}$ is a cyclic group of order ${2(d_m-1)}$. Since $\dim \fP_m (\widehat{X})=d_m-1$, the monodromy action on $\widetilde{\fP}_m (\widehat{X})$ is non-trivial.

	\end{proof}

    It remains to consider the case when $k_m=d_m+1$, and $\widehat{\OO}$ is a $2$-fold cover of $\OO$. This is possible if and only if $\fg$ is of type $C$, and $\alpha_m$ and $\alpha_{m+1}$ both appear with multiplicity $1$. 
    \begin{itemize}
        \item If $(\alpha, \beta)$ is of type (a), then $d_m=1$. The singularity of $\OO_m\subset \Spec(\CC[\OO])$ is of type $A_1$, and therefore its preimage under the map $\mu: \widehat{X}\to \Spec(\CC[\OO])$ is either a union of two singularities of type $A_1$, or is isomorphic to $\CC^2$. Since $\dim \fP_m (\widehat{X})=2$, it must be the former, and the two preimages of a point $x\in \OO_m$ must be in different connected components of $\mu^{-1}(\OO_m)$, thus $\mu^{-1}(\OO_m)$ is a disconnected $2$-fold cover of $\OO_m$.
        \item If $(\alpha, \beta)$ is of type (b), then $\Gamma$ is a dihedral group of order $4(d_m-1)$. Assume that $\Gamma_{\widehat{X}}\subset \Gamma$ is a subgroup of index $2$. Then $\widehat{\OO}_m\simeq \OO_m$. 
        
        Since $\fP_m(\widehat{X})\simeq \fz(\fg\fl_m^{d_m+1})$, \cref{pifroml} implies that all deformation directions from $\fP_m(X)$ are common deformation directions for $\Sigma$ and $\Sigma_{\widehat{X}}$. 
        
        If $\Gamma_{\widehat{X}}$ is a cyclic group of order $2(d_m-1)$, then the singularity with symmetry $(\Sigma_{\widehat{X}}, \ZZ/2\ZZ)$ is of type $B_{d_m-1}$, see \cite[Section 6.2]{Slodowy}. Using the same notations as above, $\dim V \cap {V}_{\widehat{X}} \le d_m-1$. Since $\fg$ is of type $C$, and $\alpha_m$ and $\alpha_{m+1}$ both appear with multiplicity $1$, \cref{pifroml} implies that $\dim \fP_m (X) = d_m$. We got a contradiction. 

        If $\Gamma_{\widehat{X}}$ is a dihedral group of order $2(d_m-1)$ for $d_m=2t+1$, then $\dim \fP_m (\widehat{X})\le t+2\le d_m$, and we get a contradiction.

        Thus, $\Gamma_{\widehat{X}}=\Gamma$. It remains to show that $\widehat{\OO}_m$ is a connected cover of $\OO_m$. Indeed, otherwise, we have $\dim \fP_m (\widehat{X})\ge 2 d_m > d_m+1$. Since $\dim \fP_m(\widehat{X})=\dim V_{\widehat{X}}$, it implies that the mondromy action is trivial.
        
    \end{itemize}
	\subsection{Singularities of the affinization of the universal cover}\label{final}
	Let $\OO\subset \fg$ be a nilpotent orbit, $\widetilde{\OO}$ be the universal $G$-equivariant cover of $\OO$, $\widetilde{\OO}_m\subset \Spec(\CC[\widetilde{\OO}])$ be the preimage of the orbit $\OO_m$ under the natural map $\mu: \Spec(\CC[\widetilde{\OO}])\to \overline{\OO}$. Let $\widehat{\OO}$ be any $G$-equivariant cover of the $m$-determining cover of $\OO$, such that $\widehat{\OO}_m$ is connected, and the singularity (together with its symmetry) of $\widehat{\OO}_m\subset \Spec(\CC[\widehat{\OO}])$ is equivalent to the one for the $m$-determining cover of $\OO$. 
	\begin{lemma}\label{2-min}
		Let $\check{\OO}$ be a $2$-fold cover of $\widehat{\OO}$, and $\check{\OO}_m\subset \Spec(\CC[\check{\OO}])$ be the associated cover of $\OO_m$. Then:
		\begin{itemize}
			\item[1)] $\check{\OO}_m$ is a $2$-fold cover of $\widehat{\OO}_m$; 
			\item[2)] $\check{\OO}_m$ is a connected cover of $\widehat{\OO}_m$, and the singularity of $\check{\OO}_m$ in $\Spec(\CC[\check{\OO}])$ is equivalent to the one of $\widehat{\OO}_m$ in $\Spec(\CC[\widehat{\OO}])$.
		\end{itemize} 
	\end{lemma}
	\begin{proof}
		1) The $2$-fold cover $\check{\OO}\to \widehat{\OO}$ induces a surjective map $f: \Spec(\CC[\check{\OO}])\to \Spec(\CC[\widehat{\OO}])$. Therefore, $\check{\OO}_m$ is a cover of $\widehat{\OO}_m$. Suppose that $\check{\OO}_m\simeq \widehat{\OO}_m$. Let the singularity of $\widehat{\OO}_m$ in $\Spec(\CC[\widehat{\OO}])$ be equivalent to $\Sigma_{\widehat{X}}=\CC^2/\Gamma_{\widehat{X}}$. Then, analogously to \cref{2cov}, the singularity of $\check{\OO}_m$ in $\Spec(\CC[\check{\OO}])$ is equivalent to $\Sigma_{\check{X}}=\CC^2/\Gamma_{\check{X}}$, where $\Gamma_{\check{X}}\subset \Gamma_{\widehat{X}}$ is a subgroup of index $2$. 
  
        We note that $\check{\OO}$ is a cover of the $m$-determining cover, and $\check{\OO}$ is covered by $\widetilde{\OO}$. Therefore, $\fP_m(\check{X})\simeq \fP_m (\widehat{X})$. The singularity $\Sigma_{\widehat{X}}$ was already computed, we analyze each case separately to show a contradiction. 
        
        \begin{itemize}
            \item If $\Sigma_{\widehat{X}}$ is of type $A_1$, then $\Gamma_{\check{X}}=1$, and therefore $\fP_m(\check{X})=0$.
            \item If $\Sigma_{\widehat{X}}$ is of type $D_{d_k}$ (with trivial mondoromy), then the analysis is analogous to the situation when $k_m=d_m+1$, and $\widehat{\OO}$ is a $2$-fold cover of $\OO$ studied above.
            \item If $\Sigma_{\widehat{X}}$ is of type $C_{k}$, the analysis is analogous to the one in \cref{noD}. 
            \item If $\Sigma_{\widehat{X}}$ is of type $B_k$, then $\Gamma_{\widehat{X}}$ is a cyclic group of order $2k$, the only subgroup $\Gamma_{\check{X}}$ of index $2$ is the cyclic group of order $k$, so $\Sigma_{\check{X}}$ is of type $A_{k-1}$. Thus, $\dim \fP_{m} (\check{X})\le k-1<\dim \fP_m (\widehat{X})$.
        \end{itemize}
		2) Suppose that $\check{\OO}_m=\check{\OO}_{m,1}\sqcup \check{\OO}_{m,2}$ is disconnected, but $\widehat{\OO}_m$ is connected. Then each $\check{\OO}_{m,i}$ is isomorphic to $\widehat{\OO}_m$. We denote the corresponding partial namikawa spaces by $\fP_{m,i}(\check{X})$. Analogously to \cref{2cov}, the singularity of both copies of $\widehat{\OO}_m$ in $\Spec(\CC[\check{\OO}])$ is equivalent to $\Sigma=\CC^2/\Gamma_{\widehat{X}}$. Since $\check{X}\to \widehat{X}$ is \'etale over $\widehat{\OO}\cup \widehat{\OO}_m$, one can show that the monodromy actions of $\pi_1(\widehat{\OO}_m)$ on $\widetilde{\fP}_{m,i}(\check{X})$ and $\widetilde{\fP}_m (\widehat{X})$ coincide. Thus, $\dim \fP_m (\check{X})=2\dim \fP_{m,1} (\check{X})=2\dim \fP_m (\widehat{X})$. By \cref{pi}, $\dim \check{\fP}_m=\dim \widehat{\fP}_m$, so we get a contradiction. Therefore $\check{\OO}_m$ is connected, and the singularity of $\check{\OO}_m$ in $\Spec(\CC[\check{\OO}])$ is equivalent to the one of $\widehat{\OO}_m$ in $\Spec(\CC[\widehat{\OO}])$. Coincidence of the dimenions of partial Namikawa spaces implies that the symmetries of the two singularities are equivalent.

            3) Assume that $\widehat{\OO}_m$ is disconnected. Then $\widehat{\OO}_m$ has two connected components, and the singularity of each of them in $\widehat{X}$ is of type $A_1$. Then $k_m=2$. Since there are no non-trivial proper subgroups of $\ZZ/2\ZZ$, and $\dim \fP_m (\check{X}) = 2$, $\check{\OO}_m$ has two irreducible components, and singularity to each of them is of type $A_1$.
	\end{proof} 
	Since the universal cover $\widetilde{\OO}$ of the $m$-determining cover $\widehat{\OO}$ can be obtained by a sequence of $2$-fold covers we have the following corollary.
	\begin{cor}\label{connect}
		The cover $\widetilde{\OO}_m$ has the same number of connected components as $\widehat{\OO}_m$, and the singularity (with symmetry) of each of the connected components of $\widetilde{\OO}_m$ in $\Spec(\CC[\widetilde{\OO}])$ is equivalent to the one of a connected component of $\widehat{\OO}_m$ in $\Spec(\CC[\widehat{\OO}])$, where $\widehat{\OO}$ is an $m$-determining cover of $\OO$.
	\end{cor}
	\cref{connect} together with the computation from \cref{min} imply \cref{theorem}.

    \section{Minimal special degenerations and symplectic duality}\label{sect: duality}

    In this section we discuss how results of this paper fit in a list of observations and conjectures regarding symplectic duality. For the mathematical definition of symplectic duality we refer to \cite{BPWII}. Let $\fg$ be a simple Lie algebra, $\fg^\vee$ be its Langlands dual, and  $\cN\subset \fg$, $\cN^\vee\subset \fg^\vee$ be the respective nilpotent cones. Recall from \cite{BarbaschVogan1985} (see also \cite{Lusztig1984} and \cite{Spaltenstein}) that there is a duality map $d: \cN^\vee/G\to \cN/G$ from the set of nilpotent orbits in $\fg^\vee$ to the set of nilpotent orbit in $\fg$, that we call \emph{Barbasch-Vogan-Lusztig-Spaltenstein (BVLS) duality}. The orbits in the image of $d$ are called \emph{special}, and $d$ restricts to an order reversing bijection between special nilpotent orbits in $\fg^\vee$ and $\fg$.
    
    Let $\OO_1, \OO_2\subset \cN$ be two special nilpotent orbits, such that $\OO_1\subset \overline{\OO}_2$. Pick a point $e\in \OO_1$, and let $S(e)$ be the Slodowy slice to the point $e$. We write $S_{\OO_1}^{\OO_2}$ for the intersection $S(e)\cap \overline{\OO}_2$. The following statement is the main focus of this section.
    \begin{expect}\label{wrong expect}
        $S_{\OO_1}^{\OO_2}$ is symplectic dual to $S_{d(\OO_2)}^{d(\OO_1)}$. In particular, if $S_{\OO_1}^{\OO_2}$ is a Kleinian singularity of type A, D, E, then $S_{d(\OO_2)}^{d(\OO_1)}$ is the closure of a minimal orbit in the Lie algebra of the same type, i.e. minimal singularity $a$, $d$ or $e$.
    \end{expect}

        We note that it is well-known that this expectation is not true as stated, but it holds in many cases, and should serve as a motivation for the refined version stated below.

        In \cite{JLS} Juteau, Levy and Sommers described minimal special degenerations in the nilpotent cone. They studied the expectation above, which resulted in Theorem 10.1. As an example, we copy the Hasse diagrams for $\fs\fo(9)$ and $\fs\fp(8)$ from Section 11 of \emph{loc.cit.} Since we want to analyze the duality, we reverse the diagram for $\fs\fp(8)$ (so the top orbit is the $0$ orbit). Here $b_i^{sp}$ ($c_i^{sp}$) stands for the closure of the minimal special orbit in the Lie algebra of type B (resp. C) and rank $i$.

        \begin{tikzpicture}
	\begin{scope}[shift={(-4,0)}]
		\matrix (B4special)
		[matrix of math nodes,
		nodes in empty cells,
		nodes={outer sep=0pt,minimum size=0pt},
		column sep={1.7cm,between origins},
		row sep={1.5cm,between origins}]
		{
			&&& | (a) | [9]  \\
			&&& | (b) | [7,1^2] \\
			&&& | (c) | [5,3,1]  \\
			&&& | (d) | [5,2^2]  \\
			&& | (e) | [5,1^4] &&  | (f) |  [3^3] \\
			&&&  | (g) | [3^2,1^3]  \\
			&&&  | (h) | [3,2^2,1^2] \\
			&&&  | (i) | [3,1^6] \\
			&&&  | (j) | [1^9]\\
		};

		\draw (a) -- (b) node[midway, right] {${\color{blue}B_4}$};
		\draw (b) -- (c) node[midway, right] {${\color{blue}C_3^*}$};
		\draw (c) -- (d) node[midway, right]{${\color{blue}C_1}$};  
		\draw (d) -- (e) node[midway, left] {${\color{blue}d_2^+}$}; 
		\draw (d) -- (f) node[midway, right=1pt]{${\color{blue}B_1}$}; 
		\draw (f) -- (g) node[midway,right=.5pt] {${\color{blue}B_1}$}; 
		\draw (e) -- (g) node[midway, below left=-3pt] {${\color{blue}C_2^*}$};
		\draw (g) -- (h) node[midway, right] {${\color{blue}C_1}$};
		\draw (h) -- (i) node[midway, right] {${\color{blue}d^+_3}$};  
		\draw (i) -- (j) node[midway, right] {${\color{blue}b^{\mathrm{sp}}_4}$};
		
	\end{scope}
	
	\begin{scope}[shift={(4,0)}]
		\matrix (C4special)
		[matrix of math nodes,
		nodes in empty cells,
		nodes={outer sep=0pt,minimum size=0pt},
		column sep={1.7cm,between origins},
		row sep={1.5cm,between origins}]
		{
	       	&&&  | (j) | [1^8] \\
			&&&  | (i) |  [2^2,1^4] \\
			&&&  | (h) |  [2^4]  \\
			&&& | (g) | [3^2,1^2] \\
			&& | (e) | [4,2,1^2]   && | (f) | [3^2,2]   \\
			&&& | (d) | [4,2^2] \\
			&&& | (c) | [4^2] \\
			&&& | (b) |[6,2] \\
                &&& | (a) | [8]  \\
		};
		
		\draw (a) -- (b) node[midway, right] {${\color{blue}C_4}$};
		\draw (b) -- (c) node[midway, right] {${\color{blue}C_2}$};
		\draw (c) -- (d) node[midway, right]{${\color{blue}C_2^*}$};  
		\draw (d) -- (e) node[midway, left] {${\color{blue}C_1}$}; 
		\draw (d) -- (f) node[midway, right=1pt]{${\color{blue}C_1}$}; 
		\draw (f) -- (g) node[midway,right=.5pt] {${\color{blue}C_1}$}; 
		\draw (e) -- (g) node[midway, left=3pt] {${\color{blue}C_1}$};
		\draw (g) -- (h) node[midway, right] {${\color{blue}d_2^+}$};
		\draw (h) -- (i) node[midway, right] {${\color{blue}c^{\mathrm{sp}}_2}$};  
		\draw (i) -- (j) node[midway, right] {${\color{blue}c^{\mathrm{sp}}_4}$};

		\end{scope}
\end{tikzpicture}

        The idea behind \cref{wrong expect} is based on the expectation that the symplectic dual to the Slodowy slice $S(e)$ for $e\in \OO^\vee\subset \fg^\vee$ is related to the closure of the dual orbit $d(\OO^\vee)\subset \fg$. In \cite{LMBM} Losev, Mason-Brown and the author refined the BVLS duality map by constructing an injective map $D$ from the set of nilpotent orbits in $G^\vee$ to the set of $G$-equivariant covers of special orbits in $\fg$, such that $D(\OO^\vee)$ is a $G$-equivariant cover of $\OO^\vee$. In \emph{loc.cit.} we conjecture that the symplectic dual to the Slodowy slice $S(e)$ for $e\in \OO^\vee\subset \fg^\vee$ is $\Spec(\CC[D(\OO^\vee)])$. For a nilpotent orbit $\OO_1\subset \overline{\OO}_2$ and a $G$-equivariant cover $\widehat{\OO}_2$ of $\OO_2$ let $S_{\OO_1}^{\widehat{\OO}_2}$ be a connected component of the preimage of $S_{\OO_1}^{\OO_2}$ under the map $\Spec(\CC[\widehat{\OO}_2])\to \overline{\OO}_2$.

        \begin{expect}\label{correct expect}
            Let $\OO^\vee\subset \cN^\vee$ and $\OO\subset \cN$ be special nilpotent orbits, such that $\OO^\vee\subset \overline{d(\OO)}$, and $\OO\subset \overline{d(\OO^\vee)}$. Then $S_{\OO^\vee}^{D(\OO^\vee)}$ and $S_{\OO}^{D(\OO^\vee)}$ are symplectic dual.
        \end{expect}

        Using the results of this paper, if $\OO_1\subset \overline{\OO}_2$ is of codimension $2$, we can describe $S_{\OO_1}^{\widehat{\OO}_2}$. Below we illustrate \cref{correct expect} with a diagram, analogous to the one from \cite{JLS}, where an edge between $\OO$ and $d(\OO^\vee)$ corresponds to $S_{\OO}^{D(\OO^\vee)}$. Crucially, we ignore the symmetry acting on the singularities in question. The role of the symmetry is a focus of the ongoing work with Shilin Yu.

           \begin{tikzpicture}
	\begin{scope}[shift={(-4,0)}]
		\matrix (B4special)
		[matrix of math nodes,
		nodes in empty cells,
		nodes={outer sep=0pt,minimum size=0pt},
		column sep={1.7cm,between origins},
		row sep={1.5cm,between origins}]
		{
			&&& | (a) | [9]  \\
			&&& | (b) | [7,1^2] \\
			&&& | (c) | [5,3,1]  \\
			&&& | (d) | [5,2^2]  \\
			&& | (e) | [5,1^4] &&  | (f) |  [3^3] \\
			&&&  | (g) | [3^2,1^3]  \\
			&&&  | (h) | [3,2^2,1^2] \\
			&&&  | (i) | [3,1^6] \\
			&&&  | (j) | [1^9]\\
		};

		\draw (a) -- (b) node[midway, right] {${\color{blue} A_7}$};
		\draw (b) -- (c) node[midway, right] {${\color{blue} A_3}$};
		\draw (c) -- (d) node[midway, right]{${\color{blue} A_1 \sqcup A_1}$};  
		\draw (d) -- (e) node[midway, left] {${\color{blue} A_1\sqcup A_1}$}; 
		\draw (d) -- (f) node[midway, right=1pt]{${\color{blue} A_1}$}; 
		\draw (f) -- (g) node[midway,right=.5pt] {${\color{blue} A_1}$}; 
		\draw (e) -- (g) node[midway, below left=-3pt] {${\color{blue}A_1}$};
		\draw (g) -- (h) node[midway, right] {${\color{blue}A_1}$};
		\draw (h) -- (i) node[midway, right] {${\color{red} d_3}$};  
		\draw (i) -- (j) node[midway, right] {${\color{green} d_5}$};
		
	\end{scope}
	
	\begin{scope}[shift={(4,0)}]
		\matrix (C4special)
		[matrix of math nodes,
		nodes in empty cells,
		nodes={outer sep=0pt,minimum size=0pt},
		column sep={1.7cm,between origins},
		row sep={1.5cm,between origins}]
		{
	       	&&&  | (j) | [1^8] \\
			&&&  | (i) |  [2^2,1^4] \\
			&&&  | (h) |  [2^4]  \\
			&&& | (g) | [3^2,1^2] \\
			&& | (e) | [4,2,1^2]   && | (f) | [3^2,2]   \\
			&&& | (d) | [4,2^2] \\
			&&& | (c) | [4^2] \\
			&&& | (b) |[6,2] \\
                &&& | (a) | [8]  \\
		};
		
		\draw (a) -- (b) node[midway, right] {${\color{blue}D_5}$};
		\draw (b) -- (c) node[midway, right] {${\color{blue}D_3}$};
		\draw (c) -- (d) node[midway, right]{${\color{blue} A_1}$};  
		\draw (d) -- (e) node[midway, left] {${\color{blue} A_1}$}; 
		\draw (d) -- (f) node[midway, right=1pt]{${\color{blue} A_1}$}; 
		\draw (f) -- (g) node[midway,right=.5pt] {${\color{blue} A_1}$}; 
		\draw (e) -- (g) node[midway, left=3pt] {${\color{blue} A_1 \sqcup A_1}$};
		\draw (g) -- (h) node[midway, right] {${\color{blue}A_1\sqcup A_1}$};
		\draw (h) -- (i) node[midway, right] {${\color{red}a_3}$};  
		\draw (i) -- (j) node[midway, right] {${\color{green}a_7}$};

		\end{scope}
\end{tikzpicture}

        There blue labels are obtained using the methods of this paper, and green labels correspond to the $2$-fold covers of minimal special orbits and follow from \cite{BrylinskiKostant1994}. Red labels can be obtained using the description in \cite{JLS} with extra work. Since it requires different methods to the ones used in this paper, we do not provide proof here.

	\printbibliography[title={References}]
	
\end{document}